\documentclass{article}

\usepackage{arxiv}

\usepackage[utf8]{inputenc}
\usepackage[T1]{fontenc}
\usepackage{amsmath}
\usepackage{amsfonts}
\usepackage{amssymb}
\usepackage{amsthm}
\usepackage{mathtools}
\usepackage[colorlinks=true, linkcolor=blue]{hyperref}
\usepackage{cleveref}
\usepackage{appendix}
\usepackage{tikz}
\usepackage{pgfplots}
\pgfplotsset{compat=1.18}
\usepgfplotslibrary{groupplots}

\newcommand{\tikzexternalenable}{}
\newcommand{\tikzexternaldisable}{}
\newcommand{\tikzsetnextfilename}[1]{}

\newtheorem{theorem}{Theorem}[section]
\newtheorem{remark}[theorem]{Remark}
\newtheorem{lemma}[theorem]{Lemma}

\allowdisplaybreaks%

\numberwithin{equation}{section}

\newcommand{\mytitle}{Optimal Damping for the 1D Wave Equation Using a Single
  Damper}

\newcommand{\pp}{\mathfrak{p}}
\newcommand{\vv}{\mathfrak{g}}

\newcommand{\Dint}{\ensuremath{D_{\textnormal{int}}}}
\newcommand{\Dext}{\ensuremath{D_{\textnormal{ext}}}}

\newcommand{\bbR}{\ensuremath{\mathbb{R}}}
\newcommand{\bbC}{\ensuremath{\mathbb{C}}}

\newcommand{\Rn}{\ensuremath{\bbR^{n}}}
\newcommand{\Rm}{\ensuremath{\bbR^{m}}}
\newcommand{\Rl}{\ensuremath{\bbR^{l}}}
\newcommand{\Rnn}{\ensuremath{\bbR^{n \times n}}}
\newcommand{\Rnm}{\ensuremath{\bbR^{n \times m}}}
\newcommand{\Rln}{\ensuremath{\bbR^{l \times n}}}
\newcommand{\Rnk}{\ensuremath{\bbR^{n \times k}}}

\newcommand{\cH}{\ensuremath{\mathcal{H}}}
\newcommand{\cL}{\ensuremath{\mathcal{L}}}

\newcommand{\tran}{^{\operatorname{T}}}

\newcommand{\Ltwo}{\ensuremath{\cL_{2}}}
\newcommand{\Htwo}{\ensuremath{\cH_{2}}}
\newcommand{\Hinf}{\ensuremath{\cH_{\infty}}}

\newcommand{\imag}{\boldsymbol{\imath}}

\DeclareMathOperator{\dif}{d\!}

\DeclarePairedDelimiter{\myparen}{\lparen}{\rparen}

\DeclarePairedDelimiter{\abs}{\lvert}{\rvert}

\DeclarePairedDelimiterXPP{\normtwo}[1]{}{\lVert}{\rVert}{_{2}}{#1}
\DeclarePairedDelimiterXPP{\normF}[1]{}{\lVert}{\rVert}{_{\operatorname{F}}}{#1}
\DeclarePairedDelimiterXPP{\normHinf}[1]{}{\lVert}{\rVert}{_{\Hinf}}{#1}
\DeclarePairedDelimiterXPP{\normHtwo}[1]{}{\lVert}{\rVert}{_{\Htwo}}{#1}
\DeclarePairedDelimiterXPP{\normLinf}[1]{}{\lVert}{\rVert}{_{\Linf}}{#1}
\DeclarePairedDelimiterXPP{\normLtwo}[1]{}{\lVert}{\rVert}{_{\Ltwo}}{#1}

\DeclarePairedDelimiterXPP{\ip}[2]{}{\langle}{\rangle}{}{#1, #2}
\DeclarePairedDelimiterXPP{\ipF}[2]{}{\langle}{\rangle}{_{\operatorname{F}}}{#1, #2}
\DeclarePairedDelimiterXPP{\ipHtwo}[2]{}{\langle}{\rangle}{_{\Htwo}}{#1, #2}

\DeclarePairedDelimiterXPP{\trace}[1]{\operatorname{tr}}{\lparen}{\rparen}{}{#1}
\DeclarePairedDelimiterXPP{\myspan}[1]{\operatorname{span}}{\lbrace}{\rbrace}{}{#1}
\DeclarePairedDelimiterXPP{\vecop}[1]{\operatorname{vec}}{\lparen}{\rparen}{}{#1}
\DeclarePairedDelimiterXPP{\mydiag}[1]{\operatorname{diag}}{\lparen}{\rparen}{}{#1}
\DeclarePairedDelimiterXPP{\real}[1]{\operatorname{Re}}{\lparen}{\rparen}{}{#1}

\definecolor{mplC0}{HTML}{1F77B4}
\definecolor{mplC1}{HTML}{FF7F0E}
\definecolor{mplC2}{HTML}{2CA02C}
\definecolor{mplC3}{HTML}{D62728}
\definecolor{mplC4}{HTML}{9467BD}
\definecolor{mplC5}{HTML}{8C564B}
\definecolor{mplC6}{HTML}{E377C2}
\definecolor{mplC7}{HTML}{7F7F7F}
\definecolor{mplC8}{HTML}{BCBD22}
\definecolor{mplC9}{HTML}{17BECF}

\pgfplotscreateplotcyclelist{mpl}{
  mplC0, very thick \\
  mplC1, very thick, densely dashed \\
  mplC2, very thick, densely dotted \\
  mplC3, very thick, densely dashdotted \\
  mplC4, very thick, densely dashdotdotted \\
}
\pgfplotscreateplotcyclelist{mplmark}{
  mplC0, very thick, mark=o, every mark/.append style={solid} \\
  mplC1, very thick, densely dashed, mark=x, every mark/.append style={solid}, mark size=3 \\
  mplC2, very thick, densely dotted, mark=square, every mark/.append style={solid} \\
  mplC3, very thick, densely dashdotted, mark=diamond, every mark/.append style={solid} \\
}
\pgfplotscreateplotcyclelist{mplmarkonly}{
  draw=mplC0, fill=mplC0, only marks, mark=*, mark size=1 \\
  draw=mplC1, fill=mplC1, only marks, mark=square*, mark size=1 \\
  draw=mplC2, fill=mplC2, only marks, mark=triangle*, mark size=1.5 \\
}

\let\le\leqslant%
\let\geq\geqslant%
\definecolor{lime}{HTML}{A6CE39}
\DeclareRobustCommand{\orcidicon}{
	\begin{tikzpicture}
	\draw[lime, fill=lime] (0,0)
	circle [radius=0.16]
	node[white] {{\fontfamily{qag}\selectfont \tiny ID}};
	\draw[white, fill=white] (-0.0625,0.095)
	circle [radius=0.007];
	\end{tikzpicture}
	\hspace{-2mm}
}
\foreach \x in {A, ..., Z} {%
  \expandafter\xdef\csname orcid\x\endcsname{\noexpand\href{https://orcid.org/\csname orcidauthor\x\endcsname}
    {\noexpand\orcidicon}}
}

\title{\mytitle{}}

\usepackage{authblk}

\author[1]{Petar~Mlinari\'c\orcidA{}}
\author[2]{Serkan~Gugercin\orcidB{}\thanks{The work of Gugercin is based upon
work supported by the National Science Foundation under Grant No.\ DMS-241141.}}
\author[3]{Zoran~Tomljanovi\'c\orcidC{}}

\affil[1]{
  Department of Mathematics,
  Faculty of Science,
  University of Zagreb,
  Zagreb,
  Croatia
  (\texttt{petar.mlinaric@math.hr})
}
\affil[2]{
  Department of Mathematics and
  Division of Computational Modeling and Data Analytics,
  Academy of Data Science,
  Virginia Tech,
  Blacksburg,
  VA,
  USA
  (\texttt{gugercin@vt.edu})
}
\affil[3]{
  School of Applied Mathematics and Informatics,
  J.\ J.\ Strossmayer University of Osijek,
  Osijek,
  Croatia
  (\texttt{ztomljan@mathos.hr})
}

\renewcommand{\headeright}{Preprint}
\renewcommand{\undertitle}{Preprint}

\hypersetup{
  pdftitle={\mytitle{}},
  pdfauthor={P. Mlinari\'c, S. Gugercin, and Z. Tomljanovi\'c}
}

\begin{document}

\maketitle

\begin{abstract}
  Vibrational structures are susceptible to catastrophic failures or structural damages when external forces induce resonances or repeated unwanted oscillations. One common mitigation strategy is to use dampers to suppress these disturbances. This leads to the problem of finding optimal damper viscosities and positions for a given vibrational structure. Although extensive research exists for the case of finite-dimensional systems, optimizing damper positions remains challenging due to its discrete nature. To overcome this, we introduce a novel model for the damped wave equation (at the PDE level) with a damper of viscosity $\mathfrak{g}$ at position $\mathfrak{p}$ and develop a system-theoretic input/output-based analysis in the frequency domain. In this system-theoretic formulation, while we consider average displacement as the output, for input (forcing), we analyze two separate cases, namely, the uniform and boundary forcing. For both cases, explicit formulas are derived for the corresponding transfer functions, parametrized by $\mathfrak{p}$ and $\mathfrak{g}$. This explicit parametrization by $\mathfrak{p}$ and $\mathfrak{g}$ facilitates analyzing the optimal damping problem (at the PDE level) using norms such as the $\mathcal{H}_2$ and $\mathcal{H}_\infty$ norms. We also examine limiting cases, such as when the viscosity is very large or when no external damping is present. To illustrate our approach, we present numerical examples, compare different optimization criteria, and discuss the impact of damping parameters on the damped wave equation.
\end{abstract}

\keywords{%
  damping
optimization,
wave equation, $\Htwo$ norm, $\Hinf$  norm, transfer function

}

\section{Introduction}%
\label{sec:intro}

\begin{subequations}\label{eq:wave-pde-dirac}
  Consider the damped wave equation over a string of
  length $\ell > 0$,
  with damping coefficient $d > 0$,
  stiffness $k > 0$,
  input (forcing) $u(t) \in \Rm$,
  input vector $b(x) \in \Rm$, and
  a damper of viscosity (gain) $\vv \geq 0$ at position $\pp \in (0, \ell)$:
  \begin{align}\label{eq:wave-pde-dirac-eq}
    q_{tt}(x, t)
    + (d + \vv \delta_{\pp}(x)) q_t(x, t)
    = k q_{xx}(x, t) + b(x)\tran u(t),
    \qquad
    x \in (0, \ell),\
    t > 0,
  \end{align}
  with boundary conditions, for some $b_L, b_R \in \Rm$,
  \begin{align}
    q(0, t)
    = b_L\tran u(t), \quad
    q(\ell, t)
    = b_R\tran u(t),
    \qquad
    t > 0,
  \end{align}
  and initial conditions, for smooth enough $q_0, q_1$,
  \begin{align}
    q(x, 0)
    = q_0(x), \quad
    q_t(x, 0)
    = q_1(x),
    \qquad
    x \in [0, \ell],
  \end{align}
\end{subequations}
where $\delta_{\pp}$ is the Dirac distribution concentrated at $\pp$.
We leave the details of how we handle the Dirac distribution to
\Cref{sec:coupled}, and refer the reader to~\cite{BamRT82,Tuc98,AmmT00} for
further background.
In~\eqref{eq:wave-pde-dirac}, we refer to $d$ as \emph{internal damping} and
$\vv \delta_{\pp}$ as \emph{external damping}.
In this work, we address the problem of optimizing the external damping
in~\eqref{eq:wave-pde-dirac}, i.e.,
the damper position $\pp$ and viscosity $\vv$,
for suitably chosen optimization criteria.

The goal of damping optimization is to avoid
unwanted oscillations that may cause damage,
undesirable behavior, or
collapse of the considered structure.
Generally, damping optimization refers to minimizing the influence of the input
(viewed as a disturbance) on the output.
Different objective functions (criteria) that quantify this influence could be
considered.
Prior work on finite-dimensional systems, which we describe next,
was our motivation for considering the damped wave
equation~\eqref{eq:wave-pde-dirac-eq},
together with different optimization criteria.

\subsection{Finite-dimensional Linear Vibration Systems}

Consider the linear vibration system described by the system of differential
equations
\begin{align}
  \label{eq:finitemdk}
  M \ddot{x}(t) + D \dot{x}(t) + K x(t) & = B w(t),                    \\*
  y(t)                                  & = C_1 x(t) + C_2 \dot{x}(t),
\end{align}
where $M, K \in \Rnn$ are symmetric positive definite matrices, denoting the
mass and stiffness matrices, respectively.
The state variables that model structure displacements are described by the
vector $x(t) \in \Rn$.
The primary excitation is represented by a time-dependent vector $w(t) \in \Rm$
and matrix $B \in \Rnm$.
The output vector $y(t) \in \Rl$ represents a quantity of interest that is
obtained from the state and velocity vectors via matrices $C_1, C_2 \in \Rln$.
The damping matrix $D = \Dint + \Dext$,
where $\Dint \in \Rnn$ and $\Dext \in \Rnn$,
represent internal damping and external damping matrix, respectively.
The internal damping can be modeled, e.g., as a small multiple of the critical
damping or a small multiple of proportional damping; for more details, see,
e.g.,~\cite{BeaGT20,BenTT11,Ves11}.

Typically, the main focus in damping optimization is on the external damping
matrix $\Dext$,
which models damping positions and viscosities.
In particular, one chooses
$\Dext = U \mydiag{\vv_1, \vv_2, \ldots, \vv_k} U\tran$,
where the non-negative entries $\vv_i$, $i = 1, \ldots, k$, represent the
friction coefficients of the dampers, usually called gains or viscosities, and
the matrix $U \in \Rnk$ encodes the damping positions.
For example, in the case of a simple $n$-mass oscillator with a grounded damper
attached to the $i$th mass, one obtains $U = e_i$,
where $e_i$ is the $i$th canonical vector.
Note that discretizing the damped wave equation,
modeled by~\eqref{eq:wave-pde-dirac},
leads to a finite-dimensional linear vibration system that represents an
$n$-mass oscillator with a single damping position
(more details on the discretization are given in \Cref{sec:discrete}).

One can consider excitation by a given external force
or excitation modeled by some input disturbances.
In~\cite{Ves11,TruTV15,Tom23}, authors considered periodic excitation force and
criteria based on average displacement amplitude and average energy amplitude.
On the other hand, for the input-state-output representation, a criterion can be
defined by using standard system norms, such as the $\Htwo$ or the $\Hinf$
system norm, see,
e.g.,~\cite{BlaCGM12,BeaGT20,TomV20,NakTT19,TomBG18}.
A detailed overview of system energies can be found in the
books~\cite{Ant05,Ves11,ZhoDG96}.

For all the above-mentioned criteria for damping optimization, the objective
function is non-convex and usually with many local minima, and therefore, the
associated optimization problem typically requires a large number of objective
function evaluations.
Standard approaches for damping optimization are mainly focused on the
efficient optimization of viscosities for a large number of different damper
positions.
Specifically, when an effective approach for optimizing viscosities is
available, the optimal damping positions can be identified through a
comprehensive search of the full space.
The alternative approach
would be to use some recent results that are based on heuristic or
approximation-based approaches, see, e.g.~\cite{BenTT11Prag,KanPTT19}.
However, overall, efficient optimization of damping positions is still an open
problem.

\subsection{Prior Work on Optimal Damping for the Wave Equation}

On the PDE level, analysis of the damped wave equation was considered, e.g.,
in the papers~\cite{BamRT82,Tuc98,AmmT00,CoxH08,CoxEmbree11}
where authors modeled damping as nonuniform viscous damping
without internal damping term.
In the case of pointwise damping,
the position of the damping was modeled by the Dirac distribution,
which we also use in this paper.

\subsection{Prior Work on Finite-dimensional Homogeneous Systems}

Furthermore, depending on the particular applications,
different systems can be considered:
homogeneous systems without an external force where the system
is excited by initial conditions and
inhomogeneous systems under the influence of external force.

A principal objective in the homogeneous case is to determine the optimal
damping of the system to ensure optimal evanescence of deviations from its
equilibrium.
In that case, criteria based on the underlying eigenvalue problems
play an important role in understanding system stability, robustness, and
sensitivity.
See, e.g.,~\cite{FreL99,JakMTU23,MorMTT23} where authors
considered spectral abscissa and frequency-based damping criteria.
Moreover, energy-based criteria are also widely used;
in particular, criteria based on the average total energy were considered for
the finite-dimensional case, e.g., in
the~\cite{BenTT10,BenTT11,CoxNRV04,TruV09,Ves11}, while
energy for the wave equation was considered, e.g.,
in~\cite{Mun09,MunPP06}.

\subsection{Outline}

To overcome the challenges present in the discrete approach to the damping
optimization problem,
such as the full space-search to determine optimal positions and viscosities,
this paper models the damper position with the Dirac distribution in the wave
equation.
This, in turn, allows for simultaneous analysis for variations in damping position and
viscosity.

This paper has several novel contributions to analyze and optimize the
damped wave equation.
\begin{enumerate}
  \item We derive a novel model for the damped wave equation with a damper of
        viscosity $\vv$ at position $\pp$.
  \item We consider analysis for the two important forcing strategies in the
        wave equation, achieved by an appropriate systems-theoretic input-output
        modeling.
        The first is the uniform forcing, and the second analyzes the boundary
        forcing.
  \item For both forcing cases, we develop the analysis in the frequency domain
        and derive the associated transfer function, which is explicitly
        parameterized by the viscosity $\vv$ and position $\pp$.
        An essential benefit in both forcing cases is that our result
        simultaneously analyzes the position and viscosity variation in the
        damped wave equation, and thus facilitates analyzing the
        optimal damping problem (at the continuous PDE level) using
        systems-theoretic measures such as the $\Htwo$ and $\Hinf$ norms.
  \item Several important limit cases, including those without
        internal damping, are considered, leading to explicit formulas in terms
        of the wave equation parameters and providing further insights.
  \item We also provide a comparison between finite and infinite-dimensional
        models in the context of damping optimization.
        Although the objective functions—including standard system norms—exhibit
        similar behavior as the damping position is varied in both models, the
        discretized problem inherently provides only an approximation of the
        true objective function.
        In contrast, our approach avoids this issue by modeling the damping
        position as a continuous variable, thereby preserving accuracy and
        enabling efficient, simultaneous optimization of both viscosity and
        damping position.
\end{enumerate}
The rest of the paper is organized as follows.
In \Cref{sec:wave}, we establish notation and
state the general problem of optimal damping for the one-dimensional wave
equation using one damper.
Then in \Cref{sec:uniform,sec:boundary} we analyze the cases of uniform
forcing and boundary forcing, respectively, deriving, in both cases, the
corresponding transfer functions together with the limiting cases, and
presenting numerical results illustrating the analysis.
We give conclusions in \Cref{sec:conc}.

\section{Wave Equation}%
\label{sec:wave}

The goal of this section is first to remodel the damped wave
equation~\eqref{eq:wave-pde-dirac} with the Dirac distribution as a coupled
system and then establish the basics of the frequency domain analysis we will
use in the rest of the manuscript.

\subsection{Transforming the Wave Equation with a Dirac Distribution}%
\label{sec:coupled}

Following the discussion in, e.g.,~\cite{BamRT82,Tuc98,AmmT00,CoxH08},
we first re-model the wave equation~\eqref{eq:wave-pde-dirac} with the Dirac
distribution as two coupled wave equations without Dirac distributions.
The starting point for this analysis (see, e.g.,~\cite{AgaV85}) is to interpret the
damping term $d + \vv \delta_{\pp}$ in~\eqref{eq:wave-pde-dirac} in the
integral form:
\begin{subequations}
  \begin{align}
    \label{eq:wave-eq-1}
    \int_{x_1}^{x_2}
    \myparen*{
      q_{tt}(x, t)
      + d q_t(x, t)
    }
    \dif{x}
     & =
    k q_{x}(x_2, t)
    - k q_{x}(x_1, t)
    +
    \int_{x_1}^{x_2}
    b(x)\tran u(t)
    \dif{x},
    \\
    \intertext{for $x_1, x_2 \in (0, \pp)$ or $x_1, x_2 \in (\pp, \ell)$ and}
    \label{eq:wave-eq-2}
    \int_{x_1}^{x_2}
    \myparen{
      q_{tt}(x, t)
      + d q_t(x, t)
    }
    \dif{x}
    + \vv q_t(\pp, t)
     & =
    k q_{x}(x_2, t)
    - k q_{x}(x_1, t)
    +
    \int_{x_1}^{x_2}
    b(x)\tran u(t)
    \dif{x},
  \end{align}
\end{subequations}
for $x_1 \in (0, \pp)$ and $x_2 \in (\pp, \ell)$.
Therefore, from~\eqref{eq:wave-eq-1} we have
(e.g., by differentiating with respect to $x_2$ and setting $x_2 = x$)
\begin{align*}
  q_{tt}(x, t)
  + d q_t(x, t)
   & =
  k q_{xx}(x, t)
  + b(x)\tran u(t),
\end{align*}
for $x \in (0, \pp) \cup (\pp, \ell)$ and $t > 0$.
Then, in equation~\eqref{eq:wave-eq-2},
letting $x_1 \to \pp^-$ and $x_2 \to \pp^+$ gives
\begin{equation}\label{eq:wave-eq-jump}
  \vv q_t(\pp, t)
  = k q_{x}\myparen*{\pp^+, t}
  - k q_{x}\myparen*{\pp^-, t},
\end{equation}
where we denote
\begin{equation*}
  q_{x}\myparen*{\pp^-, t}
  = \lim_{x \to \pp^-} q_{x}(x, t)
  \quad \text{and} \quad
  q_{x}\myparen*{\pp^+, t}
  = \lim_{x \to \pp^+} q_{x}(x, t).
\end{equation*}
The condition~\eqref{eq:wave-eq-jump},
together with continuity of $q$,
gives us the interface conditions:
\begin{align*}
  q\myparen*{\pp^+, t}
  - q\myparen*{\pp^-, t}
   & = 0,                         \\
  q_{x}\myparen*{\pp^+, t}
  - q_{x}\myparen*{\pp^-, t}
   & = \frac{\vv}{k} q_t(\pp, t).
\end{align*}
Therefore, we have that~\eqref{eq:wave-pde-dirac} is equivalent to the coupled
system
\begin{subequations}\label{eq:wave-pde}
  \begin{align}
    \label{eq:wave-pde-eq}
    q_{tt}(x, t)
    + d q_t(x, t)
     & =
    k q_{xx}(x, t)
    + b(x)\tran u(t),
     & x \in (0, \pp) \cup (\pp, \ell),\ t > 0, \\*
    q\myparen*{\pp^-, t}
     & = q\myparen*{\pp^+, t},
     & t > 0,                                   \\*
    \frac{\vv}{k} q_t(\pp, t)
     & =
    q_x\myparen*{\pp^+, t}
    - q_x\myparen*{\pp^-, t},
     & t > 0,                                   \\*
    \label{eq:wave-pde-bl}
    q(0, t)
     & = b_L\tran u(t),
     & t > 0,                                   \\*
    \label{eq:wave-pde-br}
    q(\ell, t)
     & = b_R\tran u(t),
     & t > 0,                                   \\*
    q(x, 0)
     & = q_0(x),
     & x \in [0, \ell],                         \\*
    q_t(x, 0)
     & = q_1(x),
     & x \in [0, \ell].
  \end{align}
\end{subequations}

\subsection{Frequency Domain Analysis and Transfer Functions}

We will follow a system-theoretic perspective in our analysis for analyzing the
optimal damping problem for~\eqref{eq:wave-pde} by defining an input-output
pair,
transforming~\eqref{eq:wave-pde} to the frequency domain via Laplace transform,
deriving frequency-domain transfer functions, and
then by defining appropriate system norms to be minimized for the damping
optimization problem.

As is common when deriving transfer function and defining system norms,
we assume zero initial conditions, and
apply the Laplace transform to~\eqref{eq:wave-pde} to obtain
(see also~\cite{CurM09})
\begin{subequations}\label{eq:wave-pde-laplace}
  \begin{align}
    \myparen*{s^2 + d s} Q(x, s)
     & =
    k Q_{xx}(x, s)
    + b(x)\tran U(s),
     & x \in (0, \pp) \cup (\pp, \ell), \\*
    Q\myparen*{\pp^-, s}
     & =
    Q\myparen*{\pp^+, s},               \\*
    \frac{\vv s}{k} Q(\pp, s)
     & =
    Q_x\myparen*{\pp^+, s}
    - Q_x\myparen*{\pp^-, s},           \\*
    Q(0, s)
     & = b_L\tran U(s),                 \\*
    Q(\ell, s)
     & = b_R\tran U(s),
  \end{align}
\end{subequations}
where $s \in \bbC$ is the Laplace variable and
$Q, U$ are Laplace transforms of $q, u$, respectively
(assuming that the Laplace transforms exist).

Let us define \emph{the displacement transfer function} $G$, i.e.,
a function such that
\begin{equation}\label{eq:g}
  Q(x, s) = G(x, s) U(s).
\end{equation}
Thus, $G$ maps input $U$ to displacement $Q$ in the Laplace domain.

Next, let the output $y(t) \in \Rl$ be a linear function of
displacement $q(\cdot, t)$ and velocity $q_t(\cdot, t)$
(we make a specific choice in~\Cref{sec:uniform,sec:boundary}).
Let $Y$ denote the Laplace transform of $y$.
We then define \emph{the output transfer function} $H$, i.e.,
a function such that
\begin{equation}\label{eq:h}
  Y(s) = H(s) U(s).
\end{equation}
In other words, $H$ maps the input $U$ to the output $Y$ in the Laplace domain.

In~\Cref{sec:uniform,sec:boundary}, we consider two natural cases for the input,
the uniform and boundary forcing.
Then, using the average displacement as the output,
we derive explicit formulae for the displacement and output transfer functions
$G$ and $H$,
which show how the damper position $\pp$ and viscosity $\vv$ influence the
dynamics.
This explicit parametrization will then facilitate analyzing the damping
optimization problem.

The damping optimization problem will require an appropriate error measure.
In the system-theoretic formulation that we consider,
the two common norms are the $\Htwo$ and $\Hinf$ norms.
For a transfer function $H$, these norms are defined as
\begin{equation} \label{eq:H2Hinf}
  \normHtwo{H}
  = \myparen*{
    \frac{1}{2 \pi}
    \int_{-\infty}^{\infty}
    \normF{H(\imag \omega)}^2
    \dif{\omega}
  }^{1/2}, \quad
  \normHinf{H}
  = \sup_{\omega \in \bbR} \,
  \normtwo{H(\imag \omega)},
\end{equation}
where $\imag$ denotes the imaginary unit.

Although not explicitly stated in~\eqref{eq:H2Hinf},
these norms (optimality criteria) depend on the optimization parameters $\pp$
and $\vv$ as the underlying transfer function depend on these parameters.
Then, the main damping optimization problems we consider are
\emph{minimizing $\normHtwo{H}$ or $\normHinf{H}$ with respect to damper
  position $\pp$ and viscosity $\vv$.}

In the following sections,
when deriving the displacement transfer function $G$ and the output transfer
function $H$,
we use SymPy~\cite{MeuSP+17} to assist with some symbolic computations.
The code is available~\cite{MliT25}.

\subsection{Finite-dimensional Approximation to the Wave Equation}%
\label{sec:discrete}

As we mentioned in~\Cref{sec:intro},
one approach to solving the damping optimization problem
would be to discretize the underlying PDE and
then try to solve the optimization problem
using the resulting finite-dimensional approximation.
We highlighted some of the disadvantages for that optimization procedure,
such as the need for a comprehensive search of the entire space.
However, since we want to present numerical comparisons to the discretized
finite-dimensional approach
(as opposed to the continuous PDE level analysis we develop here),
we will go over this semi-discretization step.

A standard approach would be to discretize the damped wave
equation~\eqref{eq:wave-pde-eq},
without external damping,
using finite differences over a uniformly subdivided domain
$0 = x_0 < x_1 < \cdots < x_n = \ell$
to obtain the set of ordinary differential equations
\begin{equation*}
  \ddot{q}_i(t)
  + d \dot{q}_i(t)
  + k \frac{-q_{i - 1}(t) + 2 q_i(t) - q_{i + 1}(t)}{h^2}
  = b(x_i)\tran u(t),
\end{equation*}
where $h = \frac{\ell}{n}$.
After taking the boundary conditions into account,
one obtains a second-order finite-dimensional system
\begin{equation*}
  M \ddot{q}(t)
  + \Dint \dot{q}(t)
  + K q(t)
  = B u(t),
\end{equation*}
where $M, \Dint, K \in \Rnn$ and $B \in \Rn$.
Then, adding the the damper at position $x_k$ would be modeled as a modification
to the damping term, leading to the parametrized dynamics
\begin{equation*}
  M \ddot{q}(t)
  + \myparen*{\Dint + \frac{\vv}{h} e_k e_k\tran} \dot{q}(t)
  + K q(t)
  = B u(t),
\end{equation*}
where $e_k$ denotes the $k$th canonical vector.
Thus we obtain a model of the form~\eqref{eq:finitemdk}.
Then, by varying the viscosity $\vv$ and the damping position $x_k$ at all
discrete positions,
one can try to compute the optimal viscosity $\vv$ and the optimal position
$x_k$.
We finally note that a discretization similar to what we described above was
considered, for example, in~\cite{CoxEmbreeHok12,Ves11},
to study the vibrational analysis of a string.

\section{Uniform Forcing}%
\label{sec:uniform}

In this section, we consider the damped wave equation with a single uniform
forcing (as the input $u(t)$) together with homogeneous Dirichlet boundary
conditions.
This corresponds to choosing
$b \equiv 1$ in~\eqref{eq:wave-pde-eq} and
$b_L = b_R = 0$ in~\eqref{eq:wave-pde-bl} and~\eqref{eq:wave-pde-br},
respectively.
We choose the average displacement as the output, i.e., we choose $y$ as
\begin{equation}\label{eq:output-avg-disp-uni}
  y(t) = \frac{1}{\ell} \int_0^\ell q(x, t) \dif{x}.
\end{equation}
In the remainder of the section,
we first derive the transfer functions $G$ and $H$
corresponding to this choice of forcing and output (\Cref{sec:uniformtf}),
analyze the limiting cases (\Cref{sec:uniformlimiting}),
and then employ these expressions to numerically investigate the damping
optimization problem
for the $\Htwo$ and $\Hinf$ measures (\Cref{sec:uniformresults}).

\subsection{Deriving the Transfer Functions for the Case of Uniform Forcing}%
\label{sec:uniformtf}

For the choices of $b \equiv 1$, and $b_L = b_R = 0$,
together with the output $y(t)$ in~\eqref{eq:output-avg-disp-uni},
the Laplace domain system~\eqref{eq:wave-pde-laplace} becomes
\begin{subequations}\label{eq:uniform-laplace}
  \begin{align}
    \label{eq:uniform-laplace-ode}
    \myparen*{s^2 + d s} Q(x, s)
     & =
    k Q_{xx}(x, s)
    + U(s),
     & x \in (0, \pp) \cup (\pp, \ell),              \\*
    \label{eq:uniform-laplace-mid1}
    Q\myparen*{\pp^-, s}
     & =
    Q\myparen*{\pp^+, s},                            \\*
    \label{eq:uniform-laplace-mid2}
    \frac{\vv s}{k} Q(\pp, s)
     & =
    Q_x\myparen*{\pp^+, s}
    - Q_x\myparen*{\pp^-, s},                        \\*
    \label{eq:uniform-laplace-bc1}
    Q(0, s)
     & = 0,                                          \\*
    \label{eq:uniform-laplace-bc2}
    Q(\ell, s)
     & = 0,                                          \\*
    \label{eq:uniform-laplace-output}
    Y(s)
     & = \frac{1}{\ell} \int_0^\ell Q(x, s) \dif{x},
  \end{align}
\end{subequations}
for $s \in \bbC$ such that~\eqref{eq:uniform-laplace-ode} has a unique solution.
The following theorem gives explicit expressions for the transfer functions $G$
and $H$ of the above problem.
Before we state the result, we want to clarify a notational choice.
To simplify the presentation,
we define several quantities such as $z, \beta_1, \beta_2$ etc.\
in~\eqref{eq:variables}.
These quantities are indeed functions; for example, $z$ is a function of $s$.
However, instead of explicitly writing $z(s)$,
we simply write $z$ so that the expressions are much easier to read.
\begin{theorem}\label{thm:uniform-g-h}
  Consider the boundary value problem~\eqref{eq:uniform-laplace}.
  The displacement and output transfer functions are given by
  \begin{equation} \label{eq:Guniform}
    G(x, s)
    =
    \frac{1}{s^2 + d s}
    \cdot
    \begin{cases}
      1
      - \cosh(z x)
      + \frac{\beta_1}{\eta} \sinh(z x),
       & x \le \pp, \\
      1
      - \cosh(z (\ell - x))
      + \frac{\beta_2}{\eta} \sinh(z (\ell - x)),
       & x > \pp,
    \end{cases}
  \end{equation}
  and
  \begin{equation}\label{eq:Huniform}
    H(s)
    =
    \frac{1}{s^2 + d s}
    \myparen*{1 - \frac{\gamma}{z \ell \eta}},
  \end{equation}
  where
  \begin{subequations}\label{eq:variables}
    \begin{align}
      \label{z}
      z
       & = \sqrt{(s^2 + d s) / k},                   \\
      \label{beta1}
      \beta_1
       & =
      2 k z \sinh^2\myparen*{\tfrac{1}{2} z \ell}
      + 2 \vv s \sinh(z (\ell - \pp))
      \sinh^2\myparen*{\tfrac{1}{2} z \pp},          \\
      \label{beta2}
      \beta_2
       & =
      2 k z \sinh^2\myparen*{\tfrac{1}{2} z \ell}
      + 2 \vv s \sinh(z \pp)
      \sinh^2\myparen*{\tfrac{1}{2} z (\ell - \pp)}, \\
      \label{gamma}
      \gamma
       & =
      4 k z \sinh^2\myparen*{\tfrac{1}{2} z \ell}
      + 8 \vv s
      \sinh\myparen*{\tfrac{1}{2} z \ell}
      \sinh\myparen*{\tfrac{1}{2} z \pp}
      \sinh\myparen*{\tfrac{1}{2} z (\ell - \pp)},
      \text{ and}                                    \\
      \label{eta}
      \eta
       & =
      k z \sinh(z \ell)
      + \vv s \sinh(z \pp) \sinh(z (\ell - \pp)).
    \end{align}
  \end{subequations}
\end{theorem}
\begin{proof}
  From the ODE~\eqref{eq:uniform-laplace-ode},
  we conclude that the solution is of the form
  \begin{equation*}
    Q(x, s)
    =
    \frac{U(s)}{s^2 + d s}
    \cdot
    \begin{cases}
      1
      + A_1 \cosh(z x)
      + B_1 \sinh(z x),
       & x \le \pp, \\
      1
      + A_2 \cosh(z (\ell - x))
      + B_2 \sinh(z (\ell - x)),
       & x > \pp.
    \end{cases}
  \end{equation*}
  From the boundary
  conditions~\eqref{eq:uniform-laplace-mid1}--\eqref{eq:uniform-laplace-bc2},
  it follows that
  \begin{align*}
    1 + A_1
     & = 0,                      \\
    1 + A_2
     & = 0,                      \\
    1
    + \cosh(z \pp) A_1
    + \sinh(z \pp) B_1
     & =
    1
    + \cosh(z (\ell - \pp)) A_2
    + \sinh(z (\ell - \pp)) B_2, \\
    \vv s
    \myparen*{
      1
      + \cosh(z \pp) A_1
      + \sinh(z \pp) B_1
    }
     & =
    -k z
    \bigl(
    \sinh(z (\ell - \pp)) A_2
    + \cosh(z (\ell - \pp)) B_2  \\*
     & \hspace{3.5em}
    + \sinh(z \pp) A_1
    + \cosh(z \pp) B_1
    \bigr).
  \end{align*}
  Solving the above $4 \times 4$ linear system via symbolic computations,
  we obtain
  \begin{align*}
    A_1
     & =
    -1,  \\
    B_1
     & =
    \frac{
      k z \cosh(z \ell)
      - k z
      + \vv s \sinh(z (\ell - \pp)) \cosh(z \pp)
      - \vv s \sinh(z (\ell - \pp))
    }{
      k z \sinh(z \ell)
      + \vv s \sinh(z \pp) \sinh(z (\ell - \pp))
    },   \\
    A_2
     & =
    -1,  \\
    B_2
     & =
    \frac{
      k z \cosh(z \ell)
      - k z
      + \vv s \sinh(z \pp) \cosh(z (\ell - \pp))
      - \vv s \sinh(z \pp)
    }{
      k z \sinh(z \ell)
      + \vv s \sinh(z \pp) \sinh(z (\ell - \pp))
    }.
  \end{align*}
  Simplifying, we find that
  \begin{align*}
    B_1 \eta
     & =
    k z (\cosh(z \ell) - 1)
    + \vv s \sinh(z (\ell - \pp)) (\cosh(z \pp) - 1)                     \\*
     & =
    2 k z \sinh^2\myparen*{\tfrac{1}{2} z \ell}
    + 2 \vv s \sinh(z (\ell - \pp)) \sinh^2\myparen*{\tfrac{1}{2} z \pp} \\*
     & =
    \beta_1,                                                             \\
    B_2 \eta
     & =
    k z (\cosh(z \ell) - 1)
    + \vv s \sinh(z \pp) (\cosh(z (\ell - \pp)) - 1)                     \\*
     & =
    2 k z \sinh^2\myparen*{\tfrac{1}{2} z \ell}
    + 2 \vv s \sinh(z \pp) \sinh^2\myparen*{\tfrac{1}{2} z (\ell - \pp)} \\*
     & =
    \beta_2.
  \end{align*}
  Thus, we proved the expression for $G$ in~\eqref{eq:Guniform}.
  It directly follows from~\eqref{eq:uniform-laplace-output} that
  $H(s) = \frac{1}{\ell} \int_0^\ell G(x, s) \dif{x}$.
  Then, inserting~\eqref{eq:Guniform} into this formula and using symbolic
  computation, we find that
  \begin{align*}
    H(s)
     & =
    \frac{1}{\ell}
    \int_0^\ell G(x, s) \dif{x} \\
     & =
    \frac{
      z \ell \eta
      - 2 k z (\cosh(z \ell) - 1)
      - 2 \vv s (
      \sinh(z \ell)
      - \sinh(z \pp)
      - \sinh(z (\ell - \pp))
      )
    }{
      z \ell \eta
      (s^2 + d s)
    }.
  \end{align*}
  Simplifying gives
  \begin{align*}
    H(s)
     & =
    \frac{
      z \ell \eta
      - 4 k z \sinh^2\myparen*{\frac{1}{2} z \ell}
      - 4 \vv s
      \sinh\myparen*{\frac{1}{2} z \ell}
      (
      \cosh\myparen*{\frac{1}{2} z \ell}
      - \cosh\myparen*{\frac{1}{2} z \ell - z \pp}
      )
    }{
      z \ell \eta
      (s^2 + d s)
    }    \\
     & =
    \frac{
      z \ell \eta
      - 4 k z \sinh^2\myparen*{\frac{1}{2} z \ell}
      - 8 \vv s
      \sinh\myparen*{\frac{1}{2} z \ell}
      \sinh\myparen*{\frac{1}{2} z (\ell - \pp)}
      \sinh\myparen*{\frac{1}{2} z \pp}
    }{
      z \ell \eta
      (s^2 + d s)
    },
  \end{align*}
  which proves~\eqref{eq:Huniform} and thus concludes the proof.
\end{proof}
Here, we note that \Cref{thm:uniform-g-h} gives an explicit expression for the
transfer function value as a function of string parameters,
which will then enable efficient analysis of the influence of the damping
parameters $\vv$ and $\pp$ on the string behavior.
To our knowledge, this is the first result explicitly deriving the dependence of
the transfer function on point-wise external damping in the context of a damped
wave equation.
Using this result, we will employ simultaneous optimization and analysis of
damping position and viscosity parameters.
This approach will provide a new perspective and assist in overcoming challenges
encountered in the finite-dimensional cases.

\subsection{Analysis of the Limiting Cases Under Uniform Forcing}%
\label{sec:uniformlimiting}

\Cref{thm:uniform-g-h} has explicitly characterized the transfer function $H$
in terms of damping parameters.
Here, to gain further insight, we discuss several important limiting cases.
In particular, we find that $H$ has a removable singularity at $s = 0$.
Additionally,
we derive expressions for the limiting cases of no external damping
($\vv = 0$, $\pp \to 0$, or $\pp \to \ell$) and
infinite external damping ($\vv \to \infty$).
\begin{lemma}\label{thm:uniform-output-limits}
  Assume the set-up of \Cref{thm:uniform-g-h}.
  Then
  \begin{align}
    \label{eq:Hsat0uniform}
    \lim_{s \to 0} H(s)
     & =
    \frac{\ell^2}{12 k},      \\
    \label{eq:Hgat0uniform}
    H(s) \rvert_{\vv = 0}
    =
    \lim_{\substack{\pp \to 0 \\ \textnormal{or} \\ \pp \to \ell}} H(s)
     & =
    \frac{
      z \ell
      - 2 \tanh\myparen*{\frac{1}{2} z \ell}
    }{
      z \ell
      (s^2 + d s)
    },                        \\
    \label{eq:Hgatinfuniform}
    \lim_{\vv \to \infty}
    H(s)
     & =
    \frac{
      z
      \ell
      - 2 \myparen*{
        \tanh\myparen*{\frac{1}{2} z \pp}
        + \tanh\myparen*{\frac{1}{2} z (\ell - \pp)}
      }
    }{
      z \ell
      (s^2 + d s)
    }.
  \end{align}
\end{lemma}
\begin{proof}
  Note that
  \begin{equation}\label{eq:lim-h-0-uniform}
    \lim_{s \to 0}
    H(s)
    =
    \lim_{s \to 0}
    \frac{1}{z^2 k}
    \myparen*{1 - \frac{\gamma}{z \ell \eta}}
    =
    \lim_{s \to 0}
    \frac{z \ell \eta - \gamma}{z^3 k \ell \eta}
    =
    \lim_{s \to 0}
    \frac{\frac{z \ell \eta - \gamma}{z^5}}{k \ell \frac{\eta}{z^2}}.
  \end{equation}
  First, using the definitions of $z$ and $\eta$ from~\eqref{z} and~\eqref{eta},
  we obtain
  \begin{equation}\label{eq:lim-eta-z2}
    \lim_{s \to 0}
    \frac{\eta}{z^2}
    =
    \lim_{s \to 0}
    \myparen*{
      \frac{
        k \sinh(z \ell)
      }{z}
      +
      \vv s
      \cdot
      \frac{\sinh(z \pp)}{z}
      \cdot
      \frac{\sinh(z (\ell - \pp))}{z}
    }
    =
    k \ell.
  \end{equation}
  Second, from
  \begin{align*}
    \frac{z \ell \eta - \gamma}{z^5}
    = {}
     &
    \frac{
      \ell
      k z \sinh(z \ell)
      - 4 k \sinh^2\myparen*{\tfrac{1}{2} z \ell}
    }{z^4} \\
     &
    +
    \frac{
      z \ell
      \vv s \sinh(z \pp) \sinh(z (\ell - \pp))
      - 8 \vv s
      \sinh\myparen*{\tfrac{1}{2} z \ell}
      \sinh\myparen*{\tfrac{1}{2} z \pp}
      \sinh\myparen*{\tfrac{1}{2} z (\ell - \pp)}
    }{z^5}
  \end{align*}
  and, via repeated use of L'H\^opital's rule,
  \begin{align*}
    \lim_{z \to 0}
    \frac{
      \ell k z \sinh(z \ell)
      - 4 k \sinh^2\myparen*{\tfrac{1}{2} z \ell}
    }{z^4}
     & =
    \frac{k \ell^4}{12}
  \end{align*}
  and
  \begin{align*}
    \lim_{z \to 0}
    \frac{
      z \ell
      \vv s \sinh(z \pp) \sinh(z (\ell - \pp))
      - 8 \vv s
      \sinh\myparen*{\tfrac{1}{2} z \ell}
      \sinh\myparen*{\tfrac{1}{2} z \pp}
      \sinh\myparen*{\tfrac{1}{2} z (\ell - \pp)}
    }{z^5} 
     & =
    0,
  \end{align*}
  we find that
  \begin{equation}\label{eq:lim-z5}
    \lim_{s \to 0}
    \frac{z \ell \eta - \gamma}{z^5}
    =
    \frac{k \ell^4}{12}.
  \end{equation}
  Then, from~\eqref{eq:lim-h-0-uniform},~\eqref{eq:lim-eta-z2},
  and~\eqref{eq:lim-z5},
  it follows that
  \begin{align*}
    \lim_{s \to 0}
    H(s)
    =
    \frac{
      \lim\limits_{s \to 0}
      \frac{z \ell \eta - \gamma}{z^5}
    }{
      k \ell
      \lim\limits_{s \to 0}
      \frac{\eta}{z^2}
    }
    =
    \frac{\frac{k \ell^4}{12}}{k^2 \ell^2}
    =
    \frac{\ell^2}{12 k},
  \end{align*}
  which proves~\eqref{eq:Hsat0uniform}.

  In the case of no external damping
  ($\vv = 0$, $\pp \to 0$, or $\pp \to \ell$),
  we obtain
  \begin{align*}
    H(s)
     & =
    \frac{
      \ell k z^2 \sinh(z \ell)
      - 4 k z \sinh^2\myparen*{\frac{1}{2} z \ell}
    }{
      \ell k z^2 \sinh(z \ell)
      (s^2 + d s)
    }
    =
    \frac{
      2 z \ell
      \sinh\myparen*{\frac{1}{2} z \ell}
      \cosh\myparen*{\frac{1}{2} z \ell}
      - 4
      \sinh^2\myparen*{\frac{1}{2} z \ell}
    }{
      2 z \ell
      \sinh\myparen*{\frac{1}{2} z \ell}
      \cosh\myparen*{\frac{1}{2} z \ell}
      (s^2 + d s)
    }    \\
     & =
    \frac{
      z \ell
      - 2 \tanh\myparen*{\frac{1}{2} z \ell}
    }{
      z \ell
      (s^2 + d s)
    },
  \end{align*}
  which proves~\eqref{eq:Hgat0uniform}.

  Finally,
  \begin{align*}
    \lim_{\vv \to \infty}
    H(s)
     & =
    \frac{
      z
      \ell
      \sinh(z \pp)
      \sinh(z (\ell - \pp))
      - 8
      \sinh\myparen*{\frac{1}{2} z \ell}
      \sinh\myparen*{\frac{1}{2} z \pp}
      \sinh\myparen*{\frac{1}{2} z (\ell - \pp)}
    }{
      z \ell
      \sinh(z \pp)
      \sinh(z (\ell - \pp))
      (s^2 + d s)
    }    \\
     & =
    \frac{
      z
      \ell
      - \frac{
        8
        \sinh\myparen*{\frac{1}{2} z \ell}
        \sinh\myparen*{\frac{1}{2} z \pp}
        \sinh\myparen*{\frac{1}{2} z (\ell - \pp)}
      }{
        \sinh(z \pp)
        \sinh(z (\ell - \pp))
      }
    }{
      z \ell
      (s^2 + d s)
    }    \\
     & =
    \frac{
      z
      \ell
      - \frac{
        2
        \sinh\myparen*{\frac{1}{2} z \ell}
      }{
        \cosh\myparen*{\frac{1}{2} z \pp}
        \cosh\myparen*{\frac{1}{2} z (\ell - \pp)}
      }
    }{
      z \ell
      (s^2 + d s)
    }    \\
     & =
    \frac{
      z
      \ell
      - 2 \frac{
        \sinh\myparen*{\frac{1}{2} z \pp}
        \cosh\myparen*{\frac{1}{2} z (\ell - \pp)}
        + \cosh\myparen*{\frac{1}{2} z \pp}
        \sinh\myparen*{\frac{1}{2} z (\ell - \pp)}
      }{
        \cosh\myparen*{\frac{1}{2} z \pp}
        \cosh\myparen*{\frac{1}{2} z (\ell - \pp)}
      }
    }{
      z \ell
      (s^2 + d s)
    }    \\
     & =
    \frac{
      z
      \ell
      - 2 \myparen*{
        \tanh\myparen*{\frac{1}{2} z \pp}
        + \tanh\myparen*{\frac{1}{2} z (\ell - \pp)}
      }
    }{
      z \ell
      (s^2 + d s)
    },
  \end{align*}
  which proves~\eqref{eq:Hgatinfuniform}, completing the proof.
\end{proof}
\Cref{thm:uniform-output-limits},
more specifically~\eqref{eq:Hsat0uniform},
illustrates that $H(0)$ is nonzero and independent of external damping.
Therefore, the $\Hinf$ norm of $H$ has a uniform positive lower bound over all
damper viscosities and positions,
which we do not know in advance for the $\Htwo$ norm.
The limiting values of the transfer function for the boundary forcing case
will be analyzed in \Cref{sec:boundary}
where we will also analyze and compare the limiting cases of transfer functions
for uniform and boundary forcing directly in terms of string parameters.

\subsection{Numerical Results for the Uniform Forcing Case}%
\label{sec:uniformresults}

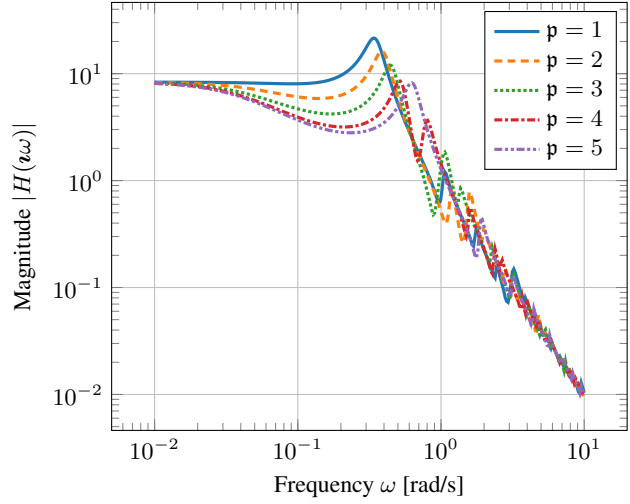
\begin{figure}[tb]
  \centering
  \tikzexternalenable%
  \tikzsetnextfilename{uniform-h-g}
  \begin{tikzpicture}
    \footnotesize
    \begin{axis}[
        grid,
        xlabel={Frequency $\omega$ [rad/s]},
        ylabel={Magnitude $\abs{H(\imag \omega)}$},
        xmode=log,
        ymode=log,
        cycle list name=mpl,
        legend entries={$\pp = 1$, $\pp = 2$, $\pp = 3$, $\pp = 4$, $\pp = 5$},
        legend cell align=left,
      ]
      \addplot table [x=w, y=H1] {fig/uniform_h_g.txt};
      \addplot table [x=w, y=H2] {fig/uniform_h_g.txt};
      \addplot table [x=w, y=H3] {fig/uniform_h_g.txt};
      \addplot table [x=w, y=H4] {fig/uniform_h_g.txt};
      \addplot table [x=w, y=H5] {fig/uniform_h_g.txt};
    \end{axis}
  \end{tikzpicture}
  \tikzexternaldisable%
  \caption{Magnitude plot of $H$ in the case of uniform forcing
    for fixed damper viscosity $\vv = 10$ and varying damping position $\pp$
    ($\ell = 10$, $d = 0.08$, $k = 1$)}%
  \label{fig:uniform-h-g}
\end{figure}

\begin{figure}[tb]
  \centering
  \tikzexternalenable%
  \tikzsetnextfilename{uniform-h-p}
  \begin{tikzpicture}
    \footnotesize
    \begin{axis}[
        grid,
        xlabel={Frequency $\omega$ [rad/s]},
        ylabel={Magnitude $\abs{H(\imag \omega)}$},
        xmode=log,
        ymode=log,
        cycle list name=mpl,
        legend entries={$\vv = 0$, $\vv = 1$, $\vv = 10$, $\vv = 100$},
        legend cell align=left,
      ]
      \addplot table [x=w, y=H1] {fig/uniform_h_p.txt};
      \addplot table [x=w, y=H2] {fig/uniform_h_p.txt};
      \addplot table [x=w, y=H3] {fig/uniform_h_p.txt};
      \addplot table [x=w, y=H4] {fig/uniform_h_p.txt};
    \end{axis}
  \end{tikzpicture}
  \tikzexternaldisable%
  \caption{Magnitude plot of $H$ in the case of uniform forcing
    for fixed damping position $\pp = 4.5$ and varying damper viscosity $\vv$
    ($\ell = 10$, $d = 0.08$, $k = 1$)}%
  \label{fig:uniform-h-p}
\end{figure}
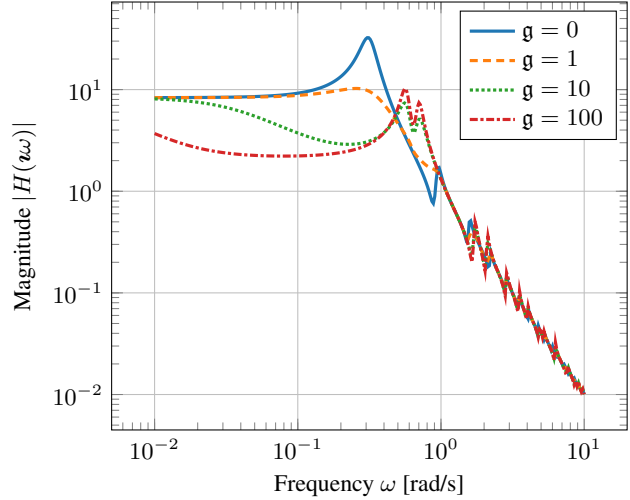

\begin{figure}[tb]
  \centering
  \tikzexternalenable%
  \tikzsetnextfilename{uniform-hinf}
  \begin{tikzpicture}
    \footnotesize
    \begin{axis}[
        enlargelimits=false,
        axis on top,
        ymode=log,
        xlabel={Position $\pp$},
        ylabel={Viscosity $\vv$},
        point meta min=8.333333333333332,
        point meta max=32.53749272254504,
        colorbar,
        colormap/viridis,
      ]
      \addplot graphics [
          xmin=0,
          xmax=5,
          ymin=1e-1,
          ymax=1e3,
        ] {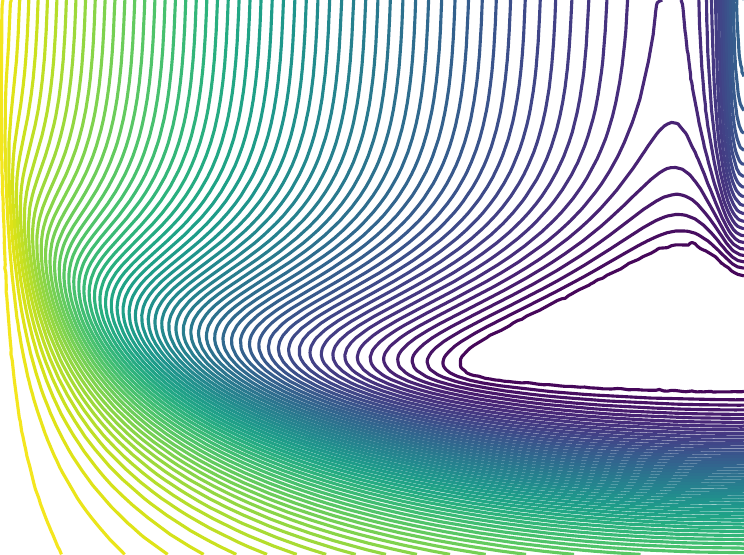};
    \end{axis}
  \end{tikzpicture}
  \tikzexternaldisable%
  \caption{$\Hinf$ norm of $H$ in the case of uniform forcing
    for different damping positions and viscosities
    ($\ell = 10$, $d = 0.08$, $k = 1$)}%
  \label{fig:uniform-hinf}
\end{figure}

\begin{figure}[tb]
  \centering
  \tikzexternalenable%
  \tikzsetnextfilename{uniform-h2}
  \begin{tikzpicture}
    \footnotesize
    \begin{axis}[
        enlargelimits=false,
        axis on top,
        ymode=log,
        xlabel={Position $\pp$},
        ylabel={Viscosity $\vv$},
        point meta min=2.419454658378959,
        point meta max=6.457985362604408,
        colorbar,
        colormap/viridis,
      ]
      \addplot graphics [
          xmin=0,
          xmax=5,
          ymin=1e-1,
          ymax=1e3,
        ] {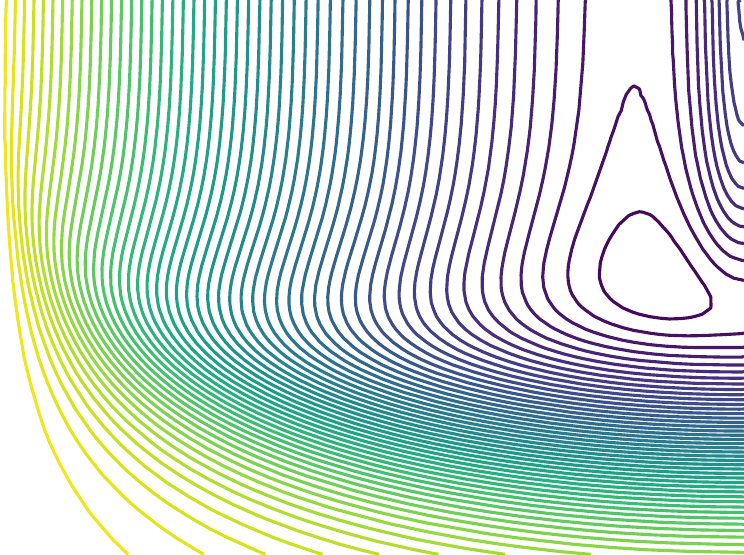};
    \end{axis}
  \end{tikzpicture}
  \tikzexternaldisable%
  \caption{$\Htwo$ norm of $H$ in the case of uniform forcing
    for different damping positions and viscosities
    ($\ell = 10$, $d = 0.08$, $k = 1$)}%
  \label{fig:uniform-h2}
\end{figure}

\begin{figure}[tb]
  \centering
  \tikzexternalenable%
  \tikzsetnextfilename{uniform-hinf-comp}
  \begin{tikzpicture}
    \footnotesize
    \begin{groupplot}[
        group style={
            group size=2 by 1,
            xlabels at=edge bottom,
            ylabels at=edge left,
            xticklabels at=edge bottom,
            yticklabels at=edge left,
            horizontal sep=3ex,
          },
        height=0.4\linewidth,
        width=0.48\linewidth,
        enlargelimits=false,
        axis on top,
        ymode=log,
        xlabel={Position $\pp$},
        ylabel={Viscosity $\vv$},
        xtick={0, 1, 2, 3, 4, 5},
        ytick={1e-1, 1, 1e1, 1e2, 1e3},
      ]
      \nextgroupplot[
        title={Continuous},
      ]
      \addplot graphics [
          xmin=0,
          xmax=5,
          ymin=1e-1,
          ymax=1e3,
        ] {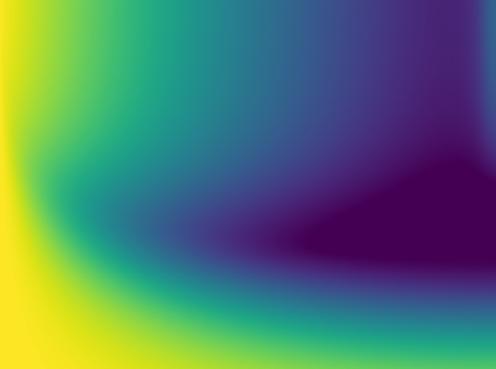};
      \nextgroupplot[
        title={Discrete},
        point meta min=8.328124999999881,
        point meta max=32.53749272254504,
        colorbar,
        colormap/viridis,
      ]
      \addplot graphics [
          xmin=0,
          xmax=5,
          ymin=1e-1,
          ymax=1e3,
        ] {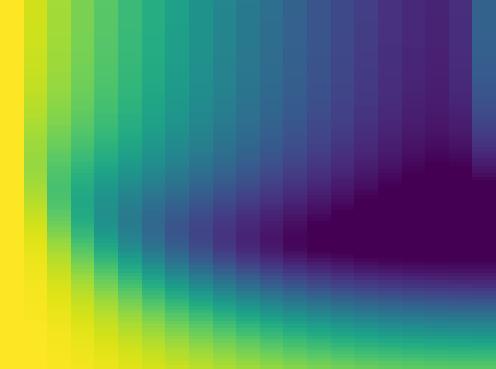};
    \end{groupplot}
  \end{tikzpicture}
  \tikzexternaldisable%
  \caption{Comparison of $\Hinf$ norm of $H$ in the case of uniform forcing
    for the continuous and discretized problem
    ($\ell = 10$, $d = 0.08$, $k = 1$)}%
  \label{fig:uniform-hinf-comp}
\end{figure}

\begin{figure}[tb]
  \centering
  \tikzexternalenable%
  \tikzsetnextfilename{uniform-h2-comp}
  \begin{tikzpicture}
    \footnotesize
    \begin{groupplot}[
        group style={
            group size=2 by 1,
            xlabels at=edge bottom,
            ylabels at=edge left,
            xticklabels at=edge bottom,
            yticklabels at=edge left,
            horizontal sep=3ex,
          },
        height=0.4\linewidth,
        width=0.48\linewidth,
        enlargelimits=false,
        axis on top,
        ymode=log,
        xlabel={Position $\pp$},
        ylabel={Viscosity $\vv$},
        xtick={0, 1, 2, 3, 4, 5},
        ytick={1e-1, 1, 1e1, 1e2, 1e3},
      ]
      \nextgroupplot[
        title={Continuous},
      ]
      \addplot graphics [
          xmin=0,
          xmax=5,
          ymin=1e-1,
          ymax=1e3,
        ] {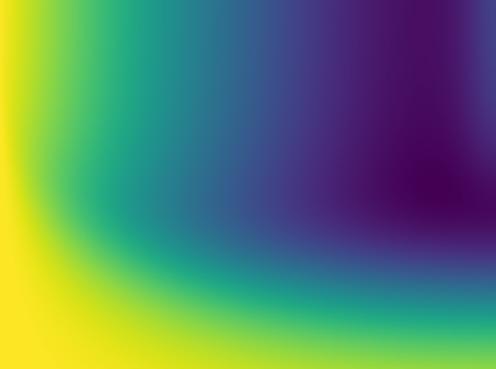};
      \nextgroupplot[
        title={Discrete},
        point meta min=2.419454658378959,
        point meta max=6.457985362604408,
        colorbar,
        colormap/viridis,
      ]
      \addplot graphics [
          xmin=0,
          xmax=5,
          ymin=1e-1,
          ymax=1e3,
        ] {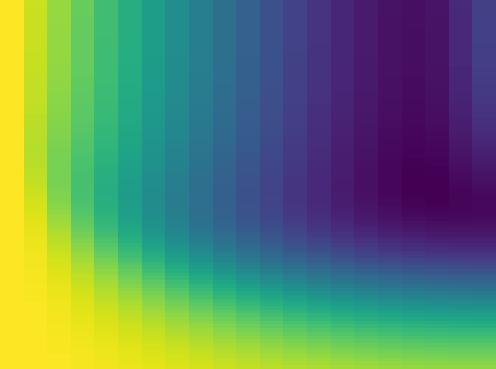};
    \end{groupplot}
  \end{tikzpicture}
  \tikzexternaldisable%
  \caption{Comparison of $\Htwo$ norm of $H$ in the case of uniform forcing
    for the continuous and discretized problem
    ($\ell = 10$, $d = 0.08$, $k = 1$)}%
  \label{fig:uniform-h2-comp}
\end{figure}

We now numerically illustrate the above analysis and show its implications for
the damping optimization problems.
We choose $\ell = 10$, $d = 0.08$, and $k = 1$.

\Cref{fig:uniform-h-g} illustrates the transfer function $H$ for $\vv = 10$ and
for five different damping positions.
We observe many peaks for higher frequencies as expected for the wave equation.
Furthermore, we also observe an almost constant value of
$\frac{\ell^2}{12 k} \approx 8.33$ close to the zero frequency,
as predicted by \Cref{thm:uniform-output-limits}.
Additionally, we note that the highest peak moves down as the damper position
approaches the middle of the string,
which is not surprising for uniform forcing.
We can also observe that for $\pp$ around $4$ and $5$,
the highest peak appears to be below the value at the zero frequency, i.e.,
the $\Hinf$ norm is achieved at zero.

Next, \Cref{fig:uniform-h-p} illustrates the transfer function $H$ for
$\pp = 4.5$ and for four different viscosities.
We can make similar observations as for the previous figure.
In addition, we observe how the transfer function values converge as
$\vv \to \infty$.
These observations and insights at the continuous PDE level are facilitated by
the explicitly derived transfer functions at the PDE level.

Now with the explicit expression of $H$,
we can directly analyze and illustrate how the $\Htwo$ and
$\Hinf$ norms of $H$ vary with respect to the position $\pp$ and
the viscosity $\vv$
as illustrated in \Cref{fig:uniform-hinf,fig:uniform-h2}.
For this example, we can first conclude that optimal parameters are different
for these two measures.
Moreover, while for the $\Htwo$ measure, the global minimum is attained at a
unique point $(\pp^*, \vv^*) \approx (4.388, 9.695)$,
for the $\Hinf$ measure, the minimum is attained over a region
(which we could expect based on \Cref{fig:uniform-h-p,fig:uniform-h-g}).
It is interesting to observe from \Cref{fig:uniform-h2,fig:uniform-hinf}
how optimal viscosity changes for the fixed parameter $\pp$
since the optimal viscosity should be much larger for smaller position
parameters.

We show the comparison between the continuous and discrete models in
\Cref{fig:uniform-hinf-comp,fig:uniform-h2-comp}.
As expected, the continuous and discrete models have similar behavior in both
cases.
However, as mentioned in the introduction, the discrete version provides limited
information concerning the damping position.
See \Cref{sec:discrete} for details on the discretization.
The simultaneous analysis of position and viscosity variation in our
approach is made possible by the continuous modeling of the position $\pp$.
While the discrete case offers only an approximation of the damped equation,
it limits the ability to perform comparable continuous analysis and
optimization of position $\pp$.

\section{Boundary Forcing}%
\label{sec:boundary}

Now we consider the damped wave equation with
a forcing on the left boundary,
homogeneous Dirichlet boundary condition on the right, and
no distributed forcing.
This corresponds to choosing
$b \equiv 0$ in~\eqref{eq:wave-pde-eq},
$b_L = 1$ in~\eqref{eq:wave-pde-bl}, and
$b_R = 0$ in~\eqref{eq:wave-pde-br}.
As in~\Cref{sec:uniform},
the output $y$ is chosen as the average displacement, i.e.,
\begin{equation}\label{eq:output-avg-disp-bnd}
  y(t) = \frac{1}{\ell} \int_0^\ell q(x, t) \dif{x}.
\end{equation}
Following the structure of~\Cref{sec:uniform},
for this choice of forcing and output,
we derive the transfer functions $G$ and $H$ (\Cref{sec:boundarytf}),
analyze the limiting cases (\Cref{sec:boundarylimiting}), and
illustrate the results for the damping optimization problem for the $\Htwo$ and
$\Hinf$ measures (\Cref{sec:boundaryresults}).
Throughout the section,
we make comparisons to the uniform forcing case where appropriate.

\subsection{Deriving the Transfer Functions for the case of Boundary Forcing}%
\label{sec:boundarytf}

Using $b \equiv 0$, $b_L = 1$, $b_R = 0$, and the output $y$
in~\eqref{eq:output-avg-disp-bnd},
the Laplace domain system~\eqref{eq:wave-pde-laplace} becomes
\begin{subequations}\label{eq:boundary-laplace}
  \begin{align}
    \label{eq:boundary-laplace-ode}
    \myparen*{s^2 + d s} Q(x, s)
     & =
    k Q_{xx}(x, s),
     & x \in (0, \pp) \cup (\pp, \ell),              \\*
    \label{eq:boundary-laplace-mid1}
    Q\myparen*{\pp^-, s}
     & =
    Q\myparen*{\pp^+, s},                            \\*
    \label{eq:boundary-laplace-mid2}
    \frac{\vv s}{k} Q(\pp, s)
     & =
    Q_x\myparen*{\pp^+, s}
    - Q_x\myparen*{\pp^-, s},                        \\*
    \label{eq:boundary-laplace-bc1}
    Q(0, s)
     & = U(s),                                       \\*
    \label{eq:boundary-laplace-bc2}
    Q(\ell, s)
     & = 0,                                          \\*
    \label{eq:boundary-laplace-output}
    Y(s)
     & = \frac{1}{\ell} \int_0^\ell Q(x, s) \dif{x},
  \end{align}
\end{subequations}
for $s \in \bbC$ for which~\eqref{eq:boundary-laplace-ode} has a unique
solution.
We derive the transfer functions $G$ and $H$ next.
\begin{theorem}\label{thm:boundary-g-h}
  Consider the boundary value problem~\eqref{eq:boundary-laplace}.
  Then, the displacement transfer function $G$ is given by
  \begin{equation}  \label{eq:Gboundary}
    G(x, s)
    =
    \begin{cases}
      \frac{k z}{\eta} \sinh(z (\ell - x))
      + \frac{\vv s}{\eta} \sinh(z (\pp - x)) \sinh(z (\ell - \pp))
       & x \le \pp, \\
      \frac{k z}{\eta} \sinh(z (\ell - x)),
       & x > \pp
    \end{cases}
  \end{equation}
  and the output transfer function $H$ by
  \begin{equation}  \label{eq:Hboundary}
    H(s) = \frac{\beta_1}{z \ell \eta},
  \end{equation}
  with $z$, $\beta_1$, and $\eta$ as in \Cref{thm:uniform-g-h}.
\end{theorem}
\begin{proof}
  From the ODE~\eqref{eq:boundary-laplace-ode},
  we conclude that the solution is of the form
  \begin{equation*}
    Q(x, s)
    =
    U(s)
    \cdot
    \begin{cases}
      A_1 \cosh(z x)
      + B_1 \sinh(z x),
       & x \le \pp, \\
      A_2 \cosh(z x)
      + B_2 \sinh(z x),
       & x > \pp.
    \end{cases}
  \end{equation*}
  From the boundary conditions, it follows that
  \begin{align*}
    A_1
     & = 1,              \\
    \cosh(z \ell) A_2
    + \sinh(z \ell) B_2
     & = 0,              \\
    \cosh(z \pp) A_1
    + \sinh(z \pp) B_1
     & =
    \cosh(z \pp) A_2
    + \sinh(z \pp) B_2,  \\
    \frac{\vv s}{k}
    \myparen*{
      \cosh(z \pp) A_1
      + \sinh(z \pp) B_1
    }
     & =
    z \sinh(z \pp) A_2
    + z \cosh(z \pp) B_2 \\*
     & \quad
    - z \sinh(z \pp) A_1
    - z \cosh(z \pp) B_1.
  \end{align*}
  Solving the above $4 \times 4$ system symbolically, we obtain
  \begin{align*}
    A_1
     & =
    1,   \\
    B_1
     & =
    -\frac{
      k z \cosh(z \ell)
      + \vv s \cosh(z \pp) \sinh(z (\ell - \pp))
    }{
      k z \sinh(z \ell)
      + \vv s \sinh(z \pp) \sinh(z (\ell - \pp))
    },   \\
    A_2
     & =
    \frac{
      k z \sinh(z \ell)
    }{
      k z \sinh(z \ell)
      + \vv s \sinh(\pp z ) \sinh(z (\ell - \pp))
    },   \\
    B_2
     & =
    -\frac{
      k z \cosh(z \ell)
    }{
      k z \sinh(z \ell)
      + \vv s \sinh(\pp z ) \sinh(z (\ell - \pp))
    }.
  \end{align*}
  Thus, for $x \le \pp$, we obtain
  \begin{align*}
    G(x, s)
     & =
    \cosh(z x)
    - \frac{
      k z \cosh(z \ell)
      + \vv s \cosh(z \pp) \sinh(z (\ell - \pp))
    }{
      k z \sinh(z \ell)
      + \vv s \sinh(z \pp) \sinh(z (\ell - \pp))
    }
    \sinh(z x) \\
     & =
    \frac{k z}{\eta}
    \myparen*{
      \sinh(z \ell)
      \cosh(z x)
      -
      \cosh(z \ell)
      \sinh(z x)
    }          \\*
     & \quad
    + \frac{\vv s}{\eta}
    \sinh(z (\ell - \pp))
    \myparen*{
      \sinh(z \pp)
      \cosh(z x)
      -
      \cosh(z \pp)
      \sinh(z x)
    }          \\
     & =
    \frac{k z}{\eta}
    \sinh(z (\ell - x))
    + \frac{\vv s}{\eta}
    \sinh(z (\ell - \pp))
    \sinh(z (\pp - x)).
  \end{align*}
  On the other hand, for $x > \pp$, we get
  \begin{equation*}
    G(x, s)
    =
    \frac{
      k z \sinh(z \ell)
    }{
      \eta
    }
    \cosh(z x)
    - \frac{
      k z \cosh(z \ell)
    }{
      \eta
    }
    \sinh(z x)
    =
    \frac{k z}{\eta}
    \sinh(z (\ell - x)),
  \end{equation*}
  thus proving~\eqref{eq:Gboundary}.
  Then, using~\eqref{eq:boundary-laplace-output} and~\eqref{eq:Gboundary},
  we obtain
  \begin{align*}
    H(s)
     & =
    \frac{1}{\ell}
    \int_0^\ell G(x, s) \dif{x} \\
     & =
    \frac{
      k z (\cosh(z \ell) - 1)
      + \vv s (
      \sinh(z \ell)
      - \sinh(z \pp) \cosh(z (\ell - \pp))
      - \sinh(z (\ell - \pp))
      )
    }{
      z \ell (
      k z \sinh(z \ell)
      + \vv s \sinh(z \pp) \sinh(z (\ell - \pp))
      )
    }                           \\
     & =
    \frac{
      2 k z \sinh^2\myparen*{\tfrac{1}{2} z \ell}
      + 2 \vv s
      \sinh^2\myparen*{\tfrac{1}{2} z \pp}
      \sinh(z (\ell - \pp))
    }{
      z \ell (
      k z \sinh(z \ell)
      + \vv s \sinh(z \pp) \sinh(z (\ell - \pp))
      )
    },
  \end{align*}
  which proves~\eqref{eq:Hboundary} using the definitions of $z$, $\beta_1$, and
  $\eta$ from \Cref{thm:uniform-g-h}, thus concluding the proof.
\end{proof}
Similarly to the previous section, \Cref{thm:boundary-g-h} gives an
explicit formula for the transfer function value in terms of string parameters,
ensuring efficient analysis of parameter variation.

\subsection{Analysis of the Limiting Cases Under Boundary Forcing}%
\label{sec:boundarylimiting}

To further analyze the damping optimization problem and to gain further insights, we now examine the behavior of the transfer function $H$ for several limiting cases, as we did in~\Cref{sec:uniformlimiting} for uniform forcing.
\begin{lemma}\label{thm:boundary-output-limits}
  Assume the set up of \Cref{thm:boundary-g-h}.
  Then
  \begin{align}
    \label{eq:Hsat0boundary}
    \lim_{s \to 0} H(s)
     & =
    \frac{1}{2},
    \\
    \label{eq:Hgat0boundary}
    H(s) \rvert_{\vv = 0}
    =
    \lim_{\substack{\pp \to 0 \\ \textnormal{or} \\ \pp \to \ell}} H(s)
     & =
    \frac{\tanh\myparen*{\tfrac{1}{2} z \ell}}{z \ell},
    \\
    \lim_{\vv \to \infty}
    \label{eq:Hgatinfboundary}
    H(s)
     & =
    \frac{\tanh\myparen*{\tfrac{1}{2} z \pp}}{z \ell}.
  \end{align}
\end{lemma}
\begin{proof}
  Inserting the definitions of $\beta_1$ and $\eta$
  from~\eqref{eq:variables} into~\eqref{eq:Hboundary},
  we obtain
  \begin{align*}
    \lim_{s \to 0} H(s)
    =
    \lim_{s \to 0}
    \frac{2}{\ell}
    \cdot
    \frac{{\sinh\myparen*{\frac{1}{2}z\ell}}}{z}
    \cdot
    \frac{{\sinh\myparen*{\frac{1}{2}z\ell}}}{\sinh\myparen*{z\ell}}
    =
    \frac{1}{2},
  \end{align*}
  proving~\eqref{eq:Hsat0boundary}.
  Similarly, in the case of no external damping ($\vv = 0$), $\pp \to 0$, or
  $\pp \to \ell$, we directly obtain from~\eqref{eq:Hboundary} that
  \begin{align*}
    H(s)
     & =
    \frac{
      2 k z \sinh^2\myparen*{\tfrac{1}{2} z \ell}
    }{
      k z^2 \ell \sinh(z \ell)
    }
    =
    \frac{
      2
      \sinh^2\myparen*{\tfrac{1}{2} z \ell}
    }{
      2 z \ell
      \sinh\myparen*{\tfrac{1}{2} z \ell}
      \cosh\myparen*{\tfrac{1}{2} z \ell}
    }
    =
    \frac{\tanh\myparen*{\tfrac{1}{2} z \ell}}{z \ell},
  \end{align*}
  which proves~\eqref{eq:Hgat0boundary}. Finally,
  the infinite viscosity ($\vv \to \infty$) in~\eqref{eq:Hboundary} yields
  \begin{align*}
    H(s)
     & =
    \frac{
      2 s
      \sinh^2\myparen*{\tfrac{1}{2} z \pp}
      \sinh(z (\ell - \pp))
    }{
      z \ell s
      \sinh(z \pp)
      \sinh(z (\ell - \pp))
    }
    =
    \frac{
      2
      \sinh^2\myparen*{\tfrac{1}{2} z \pp}
    }{
      2 z \ell
      \sinh\myparen*{\tfrac{1}{2} z \pp}
      \cosh\myparen*{\tfrac{1}{2} z \pp}
    }
    =
    \frac{\tanh\myparen*{\tfrac{1}{2} z \pp}}{z \ell},
  \end{align*}
  completing the proof.
\end{proof}

\begin{remark}
  \Cref{thm:boundary-output-limits} proves that at zero frequency,
  in the boundary forcing case,
  the transfer function value is independent of the string parameters
  $\ell$ and $k$.
  This is in contrast to the uniform case forcing case
  shown in \Cref{thm:uniform-output-limits}
  where the transfer function value at zero frequency
  depends on these parameters.
  Specifically, in terms of the stiffness $k$,
  while in the boundary forcing case
  this limit does not depend on $k$,
  in the uniform forcing case,
  systems with a smaller stiffness $k$
  have larger transfer function values at the origin.

\end{remark}

\subsection{Numerical Results}%
\label{sec:boundaryresults}

As we did in~\Cref{sec:uniformresults} for the uniform forcing case,
we now illustrate the theoretical analysis for the boundary forcing case
numerically.
We choose the same string parameters, that is,
$\ell = 10$, $d = 0.08$, and $k = 1$,
leading to the results
\Cref{fig:boundary-h-g,fig:boundary-h-p,fig:boundary-hinf,fig:boundary-h2,%
  fig:boundary-hinf-comp,fig:boundary-h2-comp}.
Before we get into specific results,
we immediately note that despite choosing the same string parameters,
the plots and results for the boundary forcing case are significantly different
from those for the uniform forcing case.
This analysis and comparison were only possible due to the novel explicit
transfer function formulation we derived in this paper.

More specifically, \Cref{fig:boundary-h-g,fig:boundary-h-p} illustrate the
output transfer function $H$ for different damping positions and different
viscosities.
Compared to the uniform case results shown
in~\Cref{fig:uniform-h-g,fig:uniform-h-p},
we observe even denser peaks that change under parameter variation.
We again observe a fixed value at the zero frequency,
demonstrating~\eqref{eq:Hsat0boundary}.
Furthermore, we note that for some choice of $\pp$ and $\vv$,
the largest value of $H(\imag \omega)$ is achieved at the zero frequency,
from which we can expect that the $\Hinf$ norm will be minimized over a region
in a range of $\pp$ and $\vv$.

\begin{figure}[tb]
  \centering
  \tikzexternalenable%
  \tikzsetnextfilename{boundary-h-g}
  \begin{tikzpicture}
    \footnotesize
    \begin{axis}[
        grid,
        xlabel={Frequency $\omega$ [rad/s]},
        ylabel={Magnitude $\abs{H(\imag \omega)}$},
        xmode=log,
        ymode=log,
        cycle list name=mpl,
        legend entries={$\pp = 5$, $\pp = 3$, $\pp = 1$},
      ]
      \addplot table [x=w, y=H1] {fig/boundary_h_g.txt};
      \addplot table [x=w, y=H2] {fig/boundary_h_g.txt};
      \addplot table [x=w, y=H3] {fig/boundary_h_g.txt};
    \end{axis}
  \end{tikzpicture}
  \tikzexternaldisable%
  \caption{Magnitude plot of $H$ in the case of boundary forcing
    for fixed damper viscosity $\vv = 2$ and varying damping position $\pp$
    ($\ell = 10$, $d = 0.08$, $k = 1$)}%
  \label{fig:boundary-h-g}
\end{figure}
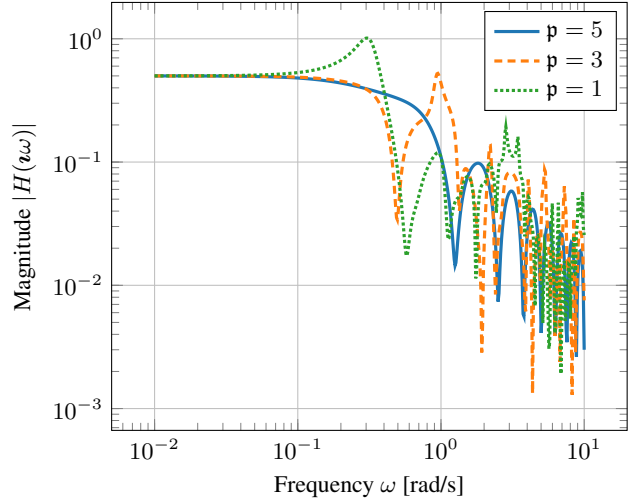

\begin{figure}[tb]
  \centering
  \tikzexternalenable%
  \tikzsetnextfilename{boundary-h-p}
  \begin{tikzpicture}
    \footnotesize
    \begin{axis}[
        grid,
        xlabel={Frequency $\omega$ [rad/s]},
        ylabel={Magnitude $\abs{H(\imag \omega)}$},
        xmode=log,
        ymode=log,
        cycle list name=mpl,
        legend entries={$\vv = 0$, $\vv = 1$, $\vv = 10$, $\vv = 100$},
      ]
      \addplot table [x=w, y=H1] {fig/boundary_h_p.txt};
      \addplot table [x=w, y=H2] {fig/boundary_h_p.txt};
      \addplot table [x=w, y=H3] {fig/boundary_h_p.txt};
      \addplot table [x=w, y=H4] {fig/boundary_h_p.txt};
    \end{axis}
  \end{tikzpicture}
  \tikzexternaldisable%
  \caption{Magnitude plot of $H$ in the case of boundary forcing
    for fixed damping position $\pp = 0.5$ and varying damper viscosity $\vv$
    ($\ell = 10$, $d = 0.08$, $k = 1$)}%
  \label{fig:boundary-h-p}
\end{figure}
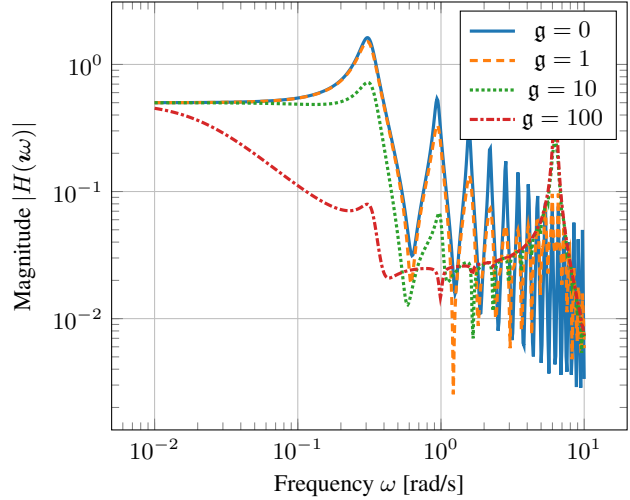

Similarly, as in \Cref{sec:uniformresults},
we show a contour plot of the $\Htwo$ and $\Hinf$ norm of $H$,
in \Cref{fig:boundary-hinf,fig:boundary-h2},
to illustrate how $H$ vary with respect to the position $\pp$ and the viscosity
$\vv$.
Here, \Cref{fig:boundary-hinf,fig:boundary-h2} illustrate the locations of
local minima and suggest that the locations of local minima and
saddle points tend to align with integer fractions of the string length, i.e.,
$\ell/2, \ell/3, \ell/4$, and so on.
Furthermore, these figures illustrate the behavior of the {optimal}
viscosity for the fixed position varies in such a way that when the damper is
closer to the forcing source,
the optimal viscosity is larger,
which is also expected from a physical perspective.

\begin{figure}[tb]
  \centering
  \tikzexternalenable%
  \tikzsetnextfilename{boundary-hinf}
  \begin{tikzpicture}
    \footnotesize
    \begin{axis}[
        enlargelimits=false,
        axis on top,
        ymode=log,
        xlabel={Position $\pp$},
        ylabel={Viscosity $\vv$},
        point meta min=0.49999999999999994,
        point meta max=1.625482043085496,
        colorbar,
        colormap/viridis,
      ]
      \addplot graphics [
          xmin=0,
          xmax=10,
          ymin=1e-1,
          ymax=1e2,
        ] {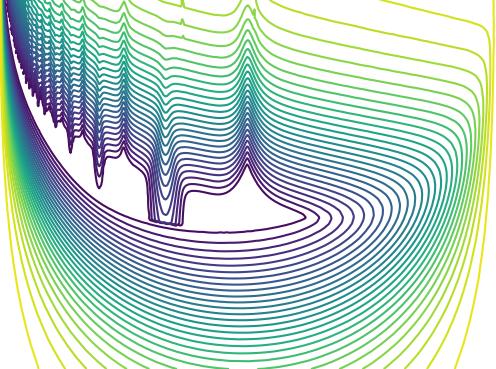};
    \end{axis}
  \end{tikzpicture}
  \tikzexternaldisable%
  \caption{$\Hinf$ norm of $H$ in the case of boundary forcing
    for different damping positions and viscosities
    ($\ell = 10$, $d = 0.08$, $k = 1$)}%
  \label{fig:boundary-hinf}
\end{figure}

\begin{figure}[tb]
  \centering
  \tikzexternalenable%
  \tikzsetnextfilename{boundary-h2}
  \begin{tikzpicture}
    \footnotesize
    \begin{axis}[
        enlargelimits=false,
        axis on top,
        ymode=log,
        xlabel={Position $\pp$},
        ylabel={Viscosity $\vv$},
        point meta min=0.20429860782436216,
        point meta max=0.3558612282088293,
        colorbar,
        colormap/viridis,
      ]
      \addplot graphics [
          xmin=0,
          xmax=10,
          ymin=1e-1,
          ymax=1e2,
        ] {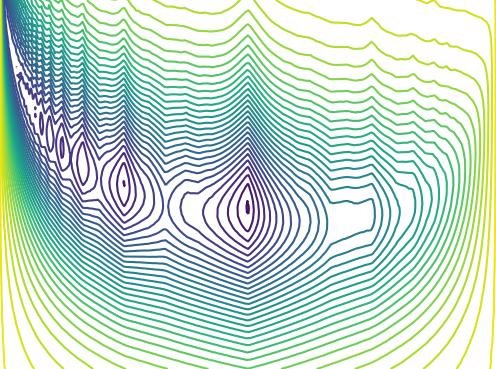};
    \end{axis}
  \end{tikzpicture}
  \tikzexternaldisable%
  \caption{$\Htwo$ norm of $H$ in the case of boundary forcing
    for different damping positions and viscosities
    ($\ell = 10$, $d = 0.08$, $k = 1$)}%
  \label{fig:boundary-h2}
\end{figure}

A comparison of the discrete and continuous cases is given in
\Cref{fig:boundary-hinf-comp,fig:boundary-h2-comp}.
{In contrast to the uniform forcing case,
which yields a unique region where the objective function is minimized,
the boundary forcing scenario results in multiple local minima.
This highlights the advantages of the continuous representation of the
position $\pp$,
particularly in both the $\Htwo$ and $\Hinf$ norms,
as the discrete model only approximates the objective function and may fail to
capture such detailed behavior.}

\begin{figure}[tb]
  \centering
  \tikzexternalenable%
  \tikzsetnextfilename{boundary-hinf-comp}
  \begin{tikzpicture}
    \footnotesize
    \begin{groupplot}[
        group style={
            group size=2 by 1,
            xlabels at=edge bottom,
            ylabels at=edge left,
            xticklabels at=edge bottom,
            yticklabels at=edge left,
            horizontal sep=3ex,
          },
        height=0.4\linewidth,
        width=0.48\linewidth,
        enlargelimits=false,
        axis on top,
        ymode=log,
        xlabel={Position $\pp$},
        ylabel={Viscosity $\vv$},
      ]
      \nextgroupplot[
        title={Continuous},
      ]
      \addplot graphics [
          xmin=0,
          xmax=10,
          ymin=1e-1,
          ymax=1e2,
        ] {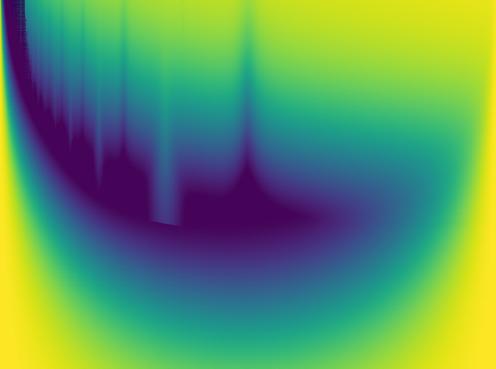};
      \nextgroupplot[
        title={Discrete},
        point meta min=0.4874999999999779,
        point meta max=1.625482043085496,
        colorbar,
        colormap/viridis,
      ]
      \addplot graphics [
          xmin=0,
          xmax=10,
          ymin=1e-1,
          ymax=1e2,
        ] {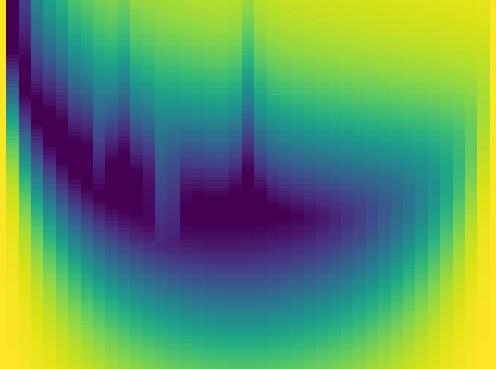};
    \end{groupplot}
  \end{tikzpicture}
  \tikzexternaldisable%
  \caption{Comparison of $\Hinf$ norm of $H$ in the case of boundary forcing
    for the continuous and discretized problem
    ($\ell = 10$, $d = 0.08$, $k = 1$)}%
  \label{fig:boundary-hinf-comp}
\end{figure}

\begin{figure}[tb]
  \centering
  \tikzexternalenable%
  \tikzsetnextfilename{boundary-h2-comp}
  \begin{tikzpicture}
    \footnotesize
    \begin{groupplot}[
        group style={
            group size=2 by 1,
            xlabels at=edge bottom,
            ylabels at=edge left,
            xticklabels at=edge bottom,
            yticklabels at=edge left,
            horizontal sep=3ex,
          },
        height=0.4\linewidth,
        width=0.48\linewidth,
        enlargelimits=false,
        axis on top,
        ymode=log,
        xlabel={Position $\pp$},
        ylabel={Viscosity $\vv$},
      ]
      \nextgroupplot[
        title={Continuous},
      ]
      \addplot graphics [
          xmin=0,
          xmax=10,
          ymin=1e-1,
          ymax=1e2,
        ] {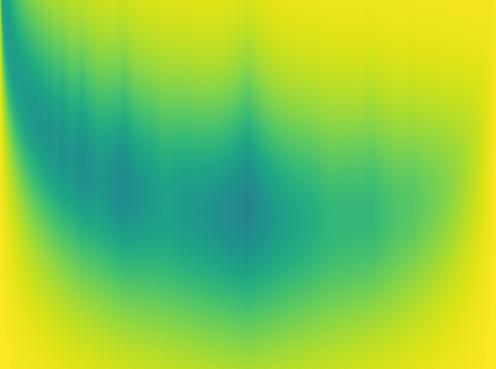};
      \nextgroupplot[
        title={Discrete},
        point meta min=0.07822890054609964,
        point meta max=0.3558612282088293,
        colorbar,
        colormap/viridis,
      ]
      \addplot graphics [
          xmin=0,
          xmax=10,
          ymin=1e-1,
          ymax=1e2,
        ] {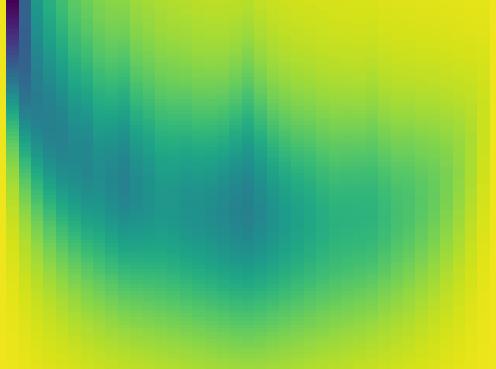};
    \end{groupplot}
  \end{tikzpicture}
  \tikzexternaldisable%
  \caption{Comparison of $\Htwo$ norm of $H$ in the case of boundary forcing
    for the continuous and discretized problem
    ($\ell = 10$, $d = 0.08$, $k = 1$)}%
  \label{fig:boundary-h2-comp}
\end{figure}

\section{Conclusions}%
\label{sec:conc}

We presented  a novel model for the damped wave equation with one damper of
viscosity $\vv$ at position $\pp$ and
developed a system-theoretic analysis in the frequency domain.
Two forcing strategies were considered from the input-output modeling
perspective: uniform and boundary forcing.
In both scenarios, the average displacement was taken as the output and
the formulae for the corresponding transfer function were derived
which explicitly illustrated the dependence on the damping parameters $\pp$ and
$\vv$.
Then, using this explicit parametrization,
we studied the optimal damping problem
by analyzing the impact of variation in the damping parameters
on the system-theoretic $\Htwo$ and $\Hinf$ optimization criteria.
We also analyzed the limiting cases,
which provide direct insight into the system's behavior
when viscosity is very large or when there is no external damping.
In addition, by treating the damping position as a continuous variable,
our model offers improved insight into the optimal position
from both numerical and physical perspectives,
compared to discrete approaches.
The impact of simultaneous variation of position and viscosity parameters is
illustrated in numerical experiments.

\bibliographystyle{alphaurl}
\bibliography{my}

\newcommand{\etalchar}[1]{$^{#1}$}
\begin{thebibliography}{MMdTT22}

\bibitem[Ant05]{Ant05}
A.~Antoulas.
\newblock {\em Approximation of Large-Scale Dynamical Systems}.
\newblock SIAM, 2005.
\newblock \href {https://doi.org/10.1137/1.9780898718713}
  {\path{doi:10.1137/1.9780898718713}}.

\bibitem[AT00]{AmmT00}
K.~Ammari and M.~Tucsnak.
\newblock Stabilization of {B}ernoulli--{E}uler beams by means of a pointwise
  feedback force.
\newblock {\em SIAM J. Control Optim.}, 39(4):1160--1181, January 2000.
\newblock \href {https://doi.org/10.1137/s0363012998349315}
  {\path{doi:10.1137/s0363012998349315}}.

\bibitem[AV85]{AgaV85}
I.~Aganovi{\'c} and K.~Veseli{\'c}.
\newblock {\em Jednad{\v{z}}be matemati{\v{c}}ke fizike}.
\newblock {\v{S}}kolska knjiga, 1985.

\bibitem[BCGM12]{BlaCGM12}
F.~Blanchini, D.~Casagrande, P.~Gardonio, and S.~Miani.
\newblock Constant and switching gains in semi-active damping of vibrating
  structures.
\newblock {\em Internat. J. Control}, 85(12):1886--1897, 2012.
\newblock \href {https://doi.org/10.1080/00207179.2012.710915}
  {\path{doi:10.1080/00207179.2012.710915}}.

\bibitem[BGT20]{BeaGT20}
C.~Beattie, S.~Gugercin, and Z.~Tomljanovi{\'c}.
\newblock Sampling-free model reduction of systems with low-rank
  parameterization.
\newblock {\em Adv. Comput. Math.}, 46(6):1--34, 2020.
\newblock \href {https://doi.org/10.1007/s10444-020-09825-8}
  {\path{doi:10.1007/s10444-020-09825-8}}.

\bibitem[BRT82]{BamRT82}
A.~Bamberger, J.~Rauch, and M.~Taylor.
\newblock A model for harmonics on stringed instruments.
\newblock {\em Archive for Rational Mechanics and Analysis}, 79(4):267--290,
  December 1982.
\newblock \href {https://doi.org/10.1007/bf00250794}
  {\path{doi:10.1007/bf00250794}}.

\bibitem[BTT11a]{BenTT11Prag}
P.~Benner, Z.~Tomljanovi{\'c}, and N.~Truhar.
\newblock {\em Damping Optimization for Linear Vibrating Systems Using
  Dimension Reduction}, pages 297--305.
\newblock Springer Netherlands, 2011.
\newblock \href {https://doi.org/10.1007/978-94-007-2069-5_41}
  {\path{doi:10.1007/978-94-007-2069-5_41}}.

\bibitem[BTT11b]{BenTT10}
P.~Benner, Z.~Tomljanovi{\'c}, and N.~Truhar.
\newblock Dimension reduction for damping optimization in linear vibrating
  systems.
\newblock {\em Z. Angew. Math. Mech.}, 91(3):179 -- 191, 2011.
\newblock \href {https://doi.org/10.1002/zamm.201000077}
  {\path{doi:10.1002/zamm.201000077}}.

\bibitem[BTT13]{BenTT11}
P.~Benner, Z.~Tomljanovi{\'c}, and N.~Truhar.
\newblock Optimal damping of selected eigenfrequencies using dimension
  reduction.
\newblock {\em Numer. Linear Algebra Appl.}, 20(1):1--17, 2013.
\newblock \href {https://doi.org/10.1002/nla.833} {\path{doi:10.1002/nla.833}}.

\bibitem[CE11]{CoxEmbree11}
Steven~J Cox and Mark Embree.
\newblock Reconstructing an even damping from a single spectrum.
\newblock {\em Inverse Problems}, 27(3):035012, February 2011.
\newblock \href {https://doi.org/10.1088/0266-5611/27/3/035012}
  {\path{doi:10.1088/0266-5611/27/3/035012}}.

\bibitem[CEH12]{CoxEmbreeHok12}
Steven~J. Cox, Mark Embree, and Jeffrey~M. Hokanson.
\newblock One can hear the composition of a string: Experiments with an inverse
  eigenvalue problem.
\newblock {\em SIAM Review}, 54(1):157--178, 2012.
\newblock \href {https://doi.org/10.1137/080731037}
  {\path{doi:10.1137/080731037}}.

\bibitem[CH08]{CoxH08}
S.~J. Cox and A.~Henrot.
\newblock Eliciting harmonics on strings.
\newblock {\em ESAIM: Control, Optimisation and Calculus of Variations},
  14(4):657--677, January 2008.
\newblock \href {https://doi.org/10.1051/cocv:2008004}
  {\path{doi:10.1051/cocv:2008004}}.

\bibitem[CM09]{CurM09}
R.~Curtain and K.~Morris.
\newblock Transfer functions of distributed parameter systems: A tutorial.
\newblock {\em Automatica}, 45(5):1101--1116, 2009.
\newblock \href {https://doi.org/10.1016/j.automatica.2009.01.008}
  {\path{doi:10.1016/j.automatica.2009.01.008}}.

\bibitem[CNRV04]{CoxNRV04}
S.~J. Cox, I.~Naki{\'c}, A.~Rittmann, and K.~Veseli{\'c}.
\newblock Lyapunov optimization of a damped system.
\newblock {\em Systems Control Lett.}, 53:187--194, 2004.
\newblock \href {https://doi.org/10.1016/j.sysconle.2004.04.004}
  {\path{doi:10.1016/j.sysconle.2004.04.004}}.

\bibitem[FL99]{FreL99}
P.~Freitas and P.~Lancaster.
\newblock On the optimal value of the spectral abscissa for a system of linear
  oscillators.
\newblock {\em SIAM J. Matrix Anal. Appl.}, 21(1):195--208, 1999.
\newblock \href {https://doi.org/10.1137/S0895479897331850}
  {\path{doi:10.1137/S0895479897331850}}.

\bibitem[JSMTU23]{JakMTU23}
N.~Jakov{\v{c}}evi{\'c}~Stor, T.~Mitchell, Z.~Tomljanovi{\'c}, and M.~Ugrica.
\newblock Fast optimization of viscosities for frequency-weighted damping of
  second-order systems.
\newblock {\em Z. Angew. Math. Mech.}, 103(5):e202100127, 2023.
\newblock \href {https://doi.org/10.1002/zamm.202100127}
  {\path{doi:10.1002/zamm.202100127}}.

\bibitem[KPTT19]{KanPTT19}
Y.~Kanno, M.~Puva{\v{c}}a, Z.~Tomljanovi{\'c}, and N.~Truhar.
\newblock Optimization of damping positions in a mechanical system.
\newblock {\em Rad HAZU, Matemati{\v{c}}ke znanosti}, 23:141--157, 2019.
\newblock \href {https://doi.org/10.21857/y26kec33q9}
  {\path{doi:10.21857/y26kec33q9}}.

\bibitem[MMdTT22]{MorMTT23}
J.~Moro, S.~Miodragovi{\'c}, F.~de~Teran, and N.~Truhar.
\newblock Frequency isolation for gyroscopic systems via hyperbolic quadratic
  eigenvalue problems.
\newblock {\em Mechanical Systems and Signal Processing}, 184(109688):1--19,
  2022.
\newblock \href {https://doi.org/10.1016/j.ymssp.2022.109688}
  {\path{doi:10.1016/j.ymssp.2022.109688}}.

\bibitem[MPP06]{MunPP06}
A.~M\"{u}nch, P.~Pedregal, and F.~Periago.
\newblock Optimal design of the damping set for the stabilization of the wave
  equation.
\newblock {\em Int. J. Differ. Equ.}, 231(1):331--358, 2006.
\newblock \href {https://doi.org/10.1016/j.jde.2006.06.009}
  {\path{doi:10.1016/j.jde.2006.06.009}}.

\bibitem[MSP{\etalchar{+}}17]{MeuSP+17}
A.~Meurer, C.~P. Smith, M.~Paprocki, O.~\v{C}ert\'{i}k, et~al.
\newblock {SymPy}: symbolic computing in {P}ython.
\newblock {\em PeerJ Computer Science}, 3:e103, January 2017.
\newblock \href {https://doi.org/10.7717/peerj-cs.103}
  {\path{doi:10.7717/peerj-cs.103}}.

\bibitem[MT25]{MliT25}
P.~Mlinari{\'c} and Z.~Tomljanovi{\'c}.
\newblock Code for "{O}ptimal {D}amping for the {1D} {W}ave {E}quation {U}sing
  a {S}ingle {D}amper".
\newblock Zenodo, 2025.
\newblock \href {https://doi.org/10.5281/zenodo.17058241}
  {\path{doi:10.5281/zenodo.17058241}}.

\bibitem[M{\"{u}}n09]{Mun09}
A.~M{\"{u}}nch.
\newblock Optimal internal dissipation of a damped wave equation using a
  topological approach.
\newblock {\em Int. J. Math. Comput. Sci.}, 19(1):15–37, 2009.
\newblock \href {https://doi.org/10.2478/v10006-009-0002-x}
  {\path{doi:10.2478/v10006-009-0002-x}}.

\bibitem[NTT19]{NakTT19}
I.~Naki{\'c}, Z.~Tomljanovi{\'c}, and N.~Truhar.
\newblock Mixed control of vibrational systems.
\newblock {\em Z. Angew. Math. Mech.}, 99(9):1--15, 2019.
\newblock \href {https://doi.org/10.1002/zamm.201800328}
  {\path{doi:10.1002/zamm.201800328}}.

\bibitem[TBG18]{TomBG18}
Z.~Tomljanovi{\'c}, C.~Beattie, and S.~Gugercin.
\newblock Damping optimization of parameter dependent mechanical systems by
  rational interpolation.
\newblock {\em Adv. Comput. Math.}, 44(6):1797--1820, 2018.
\newblock \href {https://doi.org/10.1007/s10444-018-9605-9}
  {\path{doi:10.1007/s10444-018-9605-9}}.

\bibitem[Tom23]{Tom23}
Z.~Tomljanovi{\'c}.
\newblock Damping optimization of the excited mechanical system using dimension
  reduction.
\newblock {\em Mathematics and Computers in Simulation}, 207:24--40, 2023.
\newblock \href {https://doi.org/10.1016/j.matcom.2022.12.017}
  {\path{doi:10.1016/j.matcom.2022.12.017}}.

\bibitem[TTV15]{TruTV15}
N.~Truhar, Z.~Tomljanovi{\'c}, and K.~Veseli{\'c}.
\newblock Damping optimization in mechanical systems with external force.
\newblock {\em Appl. Math. Comput.}, 250:270--279, 2015.
\newblock \href {https://doi.org/10.1016/j.amc.2014.10.081}
  {\path{doi:10.1016/j.amc.2014.10.081}}.

\bibitem[Tuc98]{Tuc98}
M.~Tucsnak.
\newblock {\em On the Pointwise Stabilization of a String}, pages 287--295.
\newblock Birkh\"{a}user Basel, 1998.
\newblock \href {https://doi.org/10.1007/978-3-0348-8849-3_22}
  {\path{doi:10.1007/978-3-0348-8849-3_22}}.

\bibitem[TV09]{TruV09}
N.~Truhar and K.~Veseli{\'c}.
\newblock An efficient method for estimating the optimal dampers' viscosity for
  linear vibrating systems using {L}yapunov equation.
\newblock {\em SIAM J. Matrix Anal. Appl.}, 31(1):18--39, 2009.
\newblock \href {https://doi.org/10.1137/070683052}
  {\path{doi:10.1137/070683052}}.

\bibitem[TV20]{TomV20}
Z.~Tomljanovi{\'c} and M.~Voigt.
\newblock Semi-active {$\mathcal{H}_\infty$}-norm damping optimization by
  adaptive interpolation.
\newblock {\em Numer. Linear Algebra Appl.}, 27(4):1--17, 2020.
\newblock \href {https://doi.org/10.1002/nla.2300}
  {\path{doi:10.1002/nla.2300}}.

\bibitem[Ves11]{Ves11}
K.~Veseli{\'c}.
\newblock {\em Damped Oscillations of Linear Systems}.
\newblock Springer Berlin Heidelberg, 2011.
\newblock \href {https://doi.org/10.1007/978-3-642-21335-9}
  {\path{doi:10.1007/978-3-642-21335-9}}.

\bibitem[ZDG96]{ZhoDG96}
K.~Zhou, J.~C. Doyle, and K.~Glover.
\newblock {\em Robust and Optimal Control}.
\newblock Prentice Hall, Upper Saddle River, NJ, 1996.
\newblock \href {https://doi.org/10.1007/978-1-4471-6257-5}
  {\path{doi:10.1007/978-1-4471-6257-5}}.

\end{thebibliography}
\addcontentsline{toc}{chapter}{References}
\end{document}